\documentclass[a4paper]{amsart}

\usepackage[mathscr]{eucal}
\usepackage[svgnames]{xcolor}
\usepackage{amsmath,amssymb,amsbsy,amsthm,amsfonts,comment,tikz}
\usepackage{hyperref}
\hypersetup{
  pdfauthor = {Morimichi Kawasaki},
  pdftitle = {From quasi-morphism to group cohomology},
  pdfpagemode = UseNone,
  bookmarksnumbered=true,
}
\usepackage{graphicx}
\usepackage[all]{xy}

\usepackage[T1]{fontenc}

\makeatletter
\@addtoreset{equation}{section} 
\makeatother

\title{On boundedness of characteristic class via quasi-morphism}

\author{Morimichi Kawasaki}
\address[Morimichi Kawasaki]{Department of Mathematical Sciences, Aoyama Gakuin University, 5-10-1 Fuchinobe, Chuo-ku, Sagamihara-shi, Kanagawa, 252-5258, Japan}
\email{kawasaki@math.aoyama.ac.jp}

\author{Shuhei Maruyama}
\address[Shuhei Maruyama]{Graduate School of Mathematics, Nagoya University, Japan}
\email{m17037h@math.nagoya-u.ac.jp}

\date{\today}

\keywords{Symplectic manifolds; the groups of Hamiltonian diffeomorphisms; quasi-morphism; group cohomology; bounded cohomology; characteristic class; Hamiltonian fibration.}

\newtheorem{theorem}{Theorem}[section]
\newtheorem{lemma}[theorem]{Lemma}
\newtheorem{proposition}[theorem]{Proposition}
\newtheorem{corollary}[theorem]{Corollary}

\theoremstyle{definition}
\newtheorem{definition}[theorem]{Definition}
\newtheorem{example}[theorem]{Example}

\newtheorem{disclaimer}[theorem]{Disclaimer}

\theoremstyle{remark}
\newtheorem{remark}[theorem]{Remark}

\newcommand{\Ham}{\mathop{\mathrm{Ham}}\nolimits}
\newcommand{\tHam}{\mathop{\widetilde{\mathrm{Ham}}}\nolimits}

\newcommand{\RR}{\mathbb{R}}

\newcommand{\ZZ}{\mathbb{Z}}

\newcommand{\AR}{A}

\newcommand{\tmu}{\mu}
\newcommand{\ke}{\mathrm{Ker}}

\newcommand{\Hom}{\mathop{\mathrm{Hom}}\nolimits}
\newcommand{\id}{\mathop{\mathrm{id}}\nolimits}
\newcommand{\Cont}{\mathop{\mathrm{Cont}}\nolimits}
\newcommand{\grp}{\mathop{\mathrm{grp}}\nolimits}
\newcommand{\Homeo}{\mathop{\mathrm{Homeo}}\nolimits}
\newcommand{\Diff}{\mathop{\mathrm{Diff}}\nolimits}

\renewcommand{\gg}{\gamma}
\newcommand{\GG}{\Gamma}
\newcommand{\tG}{\tilde{G}}

\newcommand{\tg}{\tilde{g}}
\newcommand{\Int}{\mathrm{Int}}

\newcommand{\kwsk}{\color{magenta}}
\newcommand{\maru}{\color{blue}}

\begin{document}

\begin{abstract}
  In this paper, we characterize the second bounded characteristic classes of foliated bundles in terms of the non-descendible quasi-morphisms on the universal covering of the structure group.
  As its application, we study the boundedness of obstruction classes for (contact) Hamiltonian fibrations and show the non-existence of foliated structures on some Hamiltonian fibrations. Moreover, for any closed symplectic manifold, we show the non-triviality of the second bounded cohomology group of the Hamiltonian diffeomorphism group.
\end{abstract}

\maketitle





\section{Key Theorems}\label{intro3}

Let $G$ be a connected topological group which admits the universal covering $\pi \colon \tG \to G$ and $G^{\delta}$ denote the group $G$ with the discrete topology.
Cohomology classes of the classifying spaces $BG$ and $BG^{\delta}$ are considered as universal characteristic classes of principal $G$-bundles and foliated $G$-bundles (or flat $G$-bundles), respectively.
In this paper, we concentrate our interest on the characteristic classes in degree two.
The identity homomorphism $\iota \colon G^{\delta} \to G$ induces a continuous map $B\iota \colon BG^{\delta} \to BG$ and a homomorphism
\[
  B\iota^* \colon H^2(BG;\RR) \to H^2(BG^{\delta};\RR).
\]
In this article, an element of $\mathrm{Im}(B\iota^*)$ is simply called a \textit{characteristic class of foliated $G$-bundles}.
Hence in our terminology, if a characteristic class
is non-zero for a foliated $G$-bundle $E$, the bundle $E$ is non-trivial  not only as a foliated $G$-bundle but also as a $G$-bundle.

Let $H_{\grp}^2(G;\RR)$ and $H_b^2(G;\RR)$ be the second group cohomology and second bounded cohomology of $G$, respectively.
Then, there is a canonical map
\[
  c_G \colon H_{b}^2(G;\RR) \to H_{\grp}^2(G;\RR)
\]
called the comparison map.
A group cohomology class $\alpha \in H_{\grp}^2(G;\RR)$ is called \textit{bounded} if it is in the image of $c_G$.

Since the cohomology group $H^2(BG^{\delta};\RR)$ is canonically isomorphic to $H_{\grp}^2(G;\RR)$, we can consider the intersection
\[
  \mathrm{Im}(c_G) \cap \mathrm{Im}(B\iota^*)
\]
as a subspace of $H_{\grp}^2(G;\RR)$.
This intersection $\mathrm{Im}(c_G) \cap \mathrm{Im}(B\iota^*)$ is the vector space of bounded characteristic classes of foliated $G$-bundles.

Our main theorem stated below characterizes the space $\mathrm{Im}(c_G) \cap \mathrm{Im}(B\iota^*)$ in terms of the homogeneous quasi-morphisms on the universal covering $\tG$ of $G$.
Let $Q(\tG)$ and $Q(G)$ be the vector space of all homogeneous quasi-morphisms on $\tG$ and on $G$, respectively (see Subsection \ref{subsec:prel_grp_bdd_qm} for the definition).
Let $\pi^* \colon Q(G) \to Q(\tG)$ be the pullback induced from $\pi \colon \tG \to G$.

\begin{theorem}[Theorem \ref{thm:isom_thm_qm_bdd_ch_class}]\label{thm main}
  There exists an isomorphism
  \[
    Q(\tG) / \big( H_{\grp}^1(\tG;\RR) + \pi^*Q(G) \big) \xrightarrow{\cong} \mathrm{Im}(c_G) \cap \mathrm{Im}(B\iota^*).
  \]
\end{theorem}

The following corollary, which is mainly used in applications,  immediately follows from Theorem \ref{thm main}.
\begin{corollary}[Corollary \ref{cor:isom_thm_tG_perfect_case}]\label{cor main}
  Let $G$ be a topological group whose universal covering $\tG$ satisfies $H_{\grp}^1(\tG;\RR) = 0$.
  Then there exists an isomorphism
  \[
    Q(\tG)/\pi^*Q(G) \xrightarrow{\cong} \mathrm{Im}(c_G) \cap \mathrm{Im}(B\iota^*) \  \left(\subset H_{\grp}^2(G;\RR)\right).
  \]
  In particular, if $\tmu\in Q(\tG)$ does not descend to $G$ i.e., $\tmu\notin\pi^*Q(G)$, then $\tmu$ gives rise to a non-trivial element of $H_{\grp}^2(G;\mathbb{R})$.
\end{corollary}

In the present paper, we apply the above results to the group of (contact) Hamiltonian diffeomorphisms.
As an important topic of symplectic and contact topology, many researchers have studied quasi-morphisms on these groups (for examples, see \cite{EP03}, \cite{GG}, \cite{Py6-1}, \cite{FOOO}, \cite{BZ} and \cite{FPR18}).
By combining these outcomes and our results (Theorem \ref{thm main} and Corollary \ref{cor main}), we obtain some results on the \emph{ordinary} group cohomology of these groups (see Corollaries \ref{cor bounded and unbounded} and \ref{nonflat}, Example \ref{eg1}, \ref{eg2}, \ref{eg3} and \ref{eg4}, Corollary \ref{bdd cohom on Ham shelukhin}, Proposition \ref{s3 bdl odd even} and \ref{stably nt}).

\begin{remark}
  Let $G = \Homeo_+(S^1)$ be the group of orientation preserving homeomorphisms of the circle.
  By the theorem of Thurston \cite{MR339267}, we have $H_{\grp}^2(G;\RR) = \mathrm{Im}(B\iota^*) \cong \RR \cdot e$, where $e$ is the Euler class of $\Homeo_+(S^1)$.
  It is known that the space $Q(\tG)$ is spanned by Poincar\'{e}'s rotation number $\mathrm{rot}\colon \tG \to \RR$, that is, $Q(\tG) \cong \RR \cdot \mathrm{rot}$ (see \cite{MR1876932}).
  Therefore we have $Q(\tG) \cong H_{\grp}^2(G;\RR)$.
  Note that the cohomology $H_{\grp}^2(G;\RR)$ is equal to $\mathrm{Im}(B\iota^*) \cap \mathrm{Im}(c_{G})$ since the Euler class is bounded.
  Moreover, the space $Q(\tG)$ is equal to $Q(\tG) /\left(H_{\grp}^1(\tG;\RR) + \pi^*Q(G)\right)$ since $G$ is uniformly perfect and $\tG$ is perfect.
  Thus, Theorem \ref{thm main} can be seen as a generalization of this isomorphism to an arbitrary topological group.
\end{remark}

\section{Applications to symplectic and contact geometry}\label{section applications}

We apply Corollary \ref{cor main} to symplectic and contact geometry.
A symplectic manifold $(M,\omega)$ has the natural transformation group $\Ham(M,\omega)$ called the \textit{group of Hamiltonian diffeomorphisms} \cite[DEFINITION 4.2.4.]{Ba97}, \cite[Subsection 1.2]{PR}.
A contact manifold $(M,\xi)$ also has the natural transformation group $\Cont_0(M,\xi)$ called the {\it group of contact Hamiltonian diffeomorphisms} \cite{G}.

\subsection{Boundedness of characteristic classes}\label{sec boundedness of ch class}

It is an interesting and difficult problem to determine whether a given characteristic class is bounded. 
The Milnor-Wood inequality (\cite{Mi58}, \cite{Wo71}) asserts that the Euler class of foliated $SL(2,\RR)$-bundles (and foliated $\Homeo_+(S^1)$-bundles) is bounded.
It was shown that any element of $\mathrm{Im}(B\iota^*)$ is bounded for any real algebraic subgroups of $GL(n,\RR)$ (\cite{Gr82}) and for any virtually connected Lie group with linear radical (\cite{Ch}).

As far as the authors know, for homeomorphism groups and diffeomorphism groups, in contrast, the boundedness of characteristic classes is known only for the following specific examples.
\begin{itemize}
  \item The Euler class of $\Homeo_+(S^1)$ is bounded \cite{Wo71}.
  \item The Godbillon-Vey class integrated along the fiber on $\Diff_+(S^1)$ is unbounded \cite{MR298692}.
  \item Any non-zero cohomology class of $\Homeo_0(\RR^2)$ is unbounded \cite{Cal04}.
  \item Any non-zero cohomology class of $\Homeo_0(T^2)$ is unbounded, where $T^2$ is the two-dimensional torus \cite{MR}.
  \item Some cohomology classes of $\Homeo_0(M)$ are unbounded, where $M$ is a closed Seifert-fibered $3$-manifold such that the inclusion $SO(2) \to \Homeo_0(M)$ induces an inclusion of $\pi_1(SO(2))$ as a direct factor in $\pi_1(\Homeo_0(M))$ \cite{Mann20} (see also \cite{MN21}).
\end{itemize}

Using Corollary \ref{cor main}, we show the boundedness and unboundedness of characteristic classes on (contact) Hamiltonian diffeomorphism groups.
Let us consider the symplectic manifold $(S^2 \times S^2, \omega_{\lambda})$ and the contact manifold $(S^3,\xi)$.
The symplectic form $\omega_{\lambda}$ is defined by $\omega_\lambda=\mathrm{pr}_1^\ast\omega_0+\lambda\cdot\mathrm{pr}_2^\ast\omega_0$, where $\omega_0$ is the area form on $S^2$ and $\mathrm{pr}_j\colon S^2 \times S^2 \to S^2$ is the $j$-th projection.
The contact structure $\xi$ is the standard one on $S^3$.

To simplify the notation, we set $G_{\lambda} = \Ham(S^2 \times S^2, \omega_{\lambda})$ and $H = \Cont_0(S^3, \xi)$.
For $1 < \lambda \leq 2$, we have $H^2(BG_{\lambda};\ZZ) \cong \ZZ$ and $H^2(BH;\ZZ) \cong \ZZ$ (see Section \ref{section top int}).
Let $\mathfrak{o}_{G_{\lambda}} \in H^2(BG_{\lambda};\ZZ)$ and $\mathfrak{o}_H \in H^2(BH;\ZZ)$ be the generators (or the ``primary obstruction classes with coefficients in $\ZZ$'' of $G_{\lambda}$-bundles and $H$-bundles, respectively).

Using Corollary \ref{cor main}, we can clarify the difference between these classes in terms of the boundedness.
For any $c \in H_{\grp}^2(G;\ZZ)$, let $c_{\RR} \in H_{\grp}^2(G;\RR)$ denote the corresponding cohomology class with coefficients in $\RR$.

\begin{corollary}\label{cor bounded and unbounded}
  The following properties hold.
  \begin{enumerate}
    \item The cohomology class
    \[
      B\iota^*(\mathfrak{o}_{G_{\lambda}})_{\RR} \in H_{\grp}^2(G_{\lambda};\RR)
    \]
    is bounded.
    \item The cohomology class
    \[
      B\iota^*(\mathfrak{o}_{H})_{\RR} \in H_{\grp}^2(H;\RR)
    \]
    is unbounded.
  \end{enumerate}
\end{corollary}

We will prove Corollary \ref{cor bounded and unbounded} in Section \ref{section top int}.
In order to show Corollary \ref{cor bounded and unbounded}, we use Ostrover's Calabi quasi-morphism, which is a Hamiltonian Floer theoretic invariant.

Moreover, we will show the Milnor-Wood type inequality in Section \ref{MW ineq} (Theorem \ref{MW inequality}).
Applying it to the obstruction class $(\mathfrak{o}_{G_{\lambda}})_{\RR}$, we obtain the following:
\begin{corollary}\label{nonflat}
  Let $\Sigma_h$ be a closed orientable surface of genus $h\geq 1$.
  Then, there exist infinitely many isomorphism classes of Hamiltonian fibrations over $\Sigma_h$ with the structure group $G_{\lambda} = \Ham(S^2 \times S^2, \omega_{\lambda})$ which do not admit foliated $G_{\lambda}$-bundle structures.
\end{corollary}

\begin{remark}
  The boundedness of $c$ and $c_{\RR}$ are equivalent, that is, the integer cohomology class $c$ is bounded if and only if the real cohomology class $c_{\RR}$ is bounded (this is shown by the same arguments in \cite[Lemma 29]{Ch}).
  Hence, the statement same as in Corollary \ref{cor bounded and unbounded} holds for the integer coefficients cohomology classes $B\iota^*(\mathfrak{o}_{G_{\lambda}})$ and $B\iota^*(\mathfrak{o}_{H})$.
\end{remark}

Corollaries \ref{cor bounded and unbounded} and \ref{nonflat} will be restated in more general form (see Corollary \ref{cor:bounded_unbounded_qm} and Theorem \ref{general nonflat}, respectively).

\subsection{Cohomology of (contact) Hamiltonian diffeomorphism group}\label{subsec nt coh}

Many researchers have constructed non-trivial cohomology classes of $\Ham(M,\omega)$ and $\Cont_0(M,\xi)$ (associated with the discrete topology or the $C^\infty$-topology) as characteristic classes of some (contact) Hamiltonian fibrations (\cite{R}, \cite{JK}, \cite{GKT}, \cite{Mc04}, \cite{SS}, \cite{M}, \cite{CS}).
Many non-trivial homogeneous quasi-morphisms on $\Ham(M,\omega)$ also have been obtained in several papers (\cite{BG}, \cite{EP03}, \cite{GG}, \cite{Py6-1}, \cite{Py6-2}, \cite{Mc10}, \cite[THEOREM 1.10 (1)]{FOOO}, \cite{I}, \cite{Br} \textit{et. al.}).
From these homogeneous quasi-morphisms on $\Ham(M,\omega)$, we can construct non-trivial elements of $H_b^2\left(\Ham(M,\omega);\RR\right)$ under the canonical map
\[
  \mathbf{d} \colon Q(\Ham(M,\omega)) \to H_b^2(\Ham(M,\omega);\RR)
\]
(for the map $\mathbf{d}$, see Section \ref{prel}).
Note that the classes obtained by this map are trivial as ordinary group cohomology classes in $H_{\grp}^2(\Ham(M,\omega);\RR)$.


On the other hand, using Corollary \ref{cor main} and homogeneous quasi-morphisms on the universal covering groups, we can construct non-trivial second bounded cohomology class of $\Ham(M,\omega)$ and $\Cont_0(M,\xi)$, which are also non-trivial as \emph{ordinary} cohomology classes.
Note that the universal covering $\widetilde{\Ham}(M,\omega)$ and $\widetilde{\Cont}_0(M,\xi)$ are perfect for closed symplectic and contact manifolds (\cite{Ba78}, \cite{Ry}).
Therefore these groups satisfy the assumption in Corollary \ref{cor main}.

In the following cases, Corollary \ref{cor main} provides non-trivial cohomology classes in $H_{\grp}^2(\Ham(M,\omega);\RR)$ and $H_{\grp}^2(\Cont_0(M,\xi);\RR)$.
\begin{example}\label{eg1}
Ostrover \cite{Os} constructed quasi-morphism $\mu^\lambda$ on $\widetilde{\Ham}(S^2\times S^2,\omega_\lambda)$ for $\lambda >1$. 
The homogeneous quasi-morphism $\mu^\lambda$ does not descend to $\Ham(S^2\times S^2,\omega_\lambda)$ (Proposition \ref{ostrover loop}).
Hence, we obtain a non-trivial cohomology class of $\Ham(S^2\times S^2,\omega_\lambda)$ from $\mu^\lambda$.
\end{example}

\begin{example}\label{eg3}
Ostrover and Tyomkin \cite{OT} constructed two homogeneous quasi-morphisms $\mu^1,\mu^2\colon\widetilde{\Ham}(M,\omega)\to\RR$ when $(M,\omega)$ is the $1$ points blow up of $\mathbb{C}P^2$ with some toric Fano symplectic form.
The restrictions of $\mu_1,\mu_2$ to $\pi_1\left(\Ham(M,\omega)\right)$ are linear independent.
Hence, Corollary \ref{cor main} implies the dimension of $H_{\grp}^2\left(\Ham(M,\omega);\mathbb{R}\right)$ is larger than one.
\end{example}
\begin{example}\label{eg2}
Fukaya, Oh, Ohta and Ono \cite[THEOREM 1.10 (3)]{FOOO} constructed quasi-morphisms on $\widetilde{\Ham}(M,\omega)$ when $(M,\omega)$ is the $k$ points blow up of $\mathbb{C}P^2$ with some toric symplectic form, where $k\geq2$.
Their quasi-morphisms do not descend to $\Ham(M,\omega)$ \cite[THEOREM 30.13]{FOOO}.
Hence, we can construct a non-trivial element of $H_{\grp}^2\left(\Ham(M,\omega);\mathbb{R}\right)$ from their quasi-morphisms.
\end{example}

\begin{example}\label{eg4}
Givental \cite{Gi} constructed a homogeneous quasi-morphism $\mu$ on $\widetilde{\Cont}_0(\RR P^{2n+1},\xi)$ that is called the {\it non-linear Maslov index} (see also \cite{GBS}, \cite{BZ}). This quasi-morphism $\mu$ does not descend to $\Cont_0(\RR P^{2n+1},\xi)$. Hence we obtain a non-trivial element of $H_{\grp}^2(\Cont_0(\RR P^{2n+1},\xi);\RR)$.
\end{example}

In Section \ref{section top int}, we also show the following: 

\begin{corollary}\label{bdd cohom on Ham}
Let $(M,\omega)$ be a closed symplectic manifold.
Then there exists an injective homomorphism
\[\mathfrak{d}_b\colon Q(\widetilde{\Ham}(M,\omega))\to H_b^2(\Ham(M,\omega);\mathbb{R}).\]
\end{corollary}

In \cite{Sh}, for every closed symplectic manifold $(M,\omega)$, Shelukhin constructed a non-trivial homogeneous quasi-morphism $\mu_S\colon \widetilde{\Ham}(M,\omega)\to\RR$.
Therefore, the following corollary follows from Corollary \ref{bdd cohom on Ham}.
\begin{corollary}\label{bdd cohom on Ham shelukhin}
For every closed symplectic manifold $(M,\omega)$, the bounded cohomology group $H_b^2(\Ham(M,\omega);\mathbb{R})$ is non-zero.
\end{corollary}


\begin{remark}\label{spectral invariant}
Quasi-morphisms in \cite{EP03}, \cite{Os}, \cite{OT}, \cite{Mc10}, \cite{FOOO}, \cite{U}, \cite{B}, \cite{C} and \cite{V18} are constructed via the Hamiltonian Floer theory.
As good textbooks on this topic, we refer to \cite{PR} and \cite{FOOO}.
\end{remark}

\begin{disclaimer}
Throughout the present paper, we tacitly assume that topological group $G$ is path-connected, locally path-connected, and semilocally simply-connected.
In particular, every topological group $G$ in the present paper admits the universal covering $\pi\colon\tG\to G$.
\end{disclaimer}

\subsection{Organization of the paper}

Section \ref{prel} collects preliminary facts.
Section \ref{cohom class sec} and Section \ref{diagram sec} are devoted to show an isomorphism theorem (Theorem \ref{thm:isomorphism_theorem_general}) for an arbitrary group extension.
In Section \ref{section top int}, we prove Theorem \ref{thm main} by applying the isomorphism theorem to a topological group and its universal covering.
We give applications in Section \ref{MW ineq} and Section \ref{nonex sec}.
In Section \ref{MW ineq}, we prove a Milnor-Wood type inequality and show the non-existence of foliated structures on Hamiltonian fibrations.
In Section \ref{nonex sec}, we consider an extension problem of homomorphisms on $\pi_1(G)$ to $\tG$.
In Appendix \ref{fibration sec}, we give examples of non-trivial (contact) Hamiltonian fibrations.


\section{Preliminaries}\label{prel}
\subsection{(Bounded) group cohomology and quasi-morphism}\label{subsec:prel_grp_bdd_qm}
We briefly review the (bounded) cohomology of (discrete) group and the quasi-morphism.
Let $G$ be a group and $\AR$ an abelian group.
Let $C_{\grp}^n(G ; \AR)$ denote the set of $n$-cochains $c \colon G^n \to \AR$ and $\delta \colon C_{\grp}^n(G ; \AR) \to C_{\grp}^{n+1}(G ; \AR)$ the coboundary map.
For $c \in C_{\grp}^1(G ; \AR)$, its coboundary $\delta c \in C_{\grp}^2(G ; \AR)$ is defined by
\[\delta c(g_1,g_2) = c(g_1)+c(g_2)-c(g_1g_2)\]
for $g_1,g_2 \in G$ (see \cite{Bro} for the precise definition of $\delta$).
The cohomology $H_{\grp}^\bullet(G;\AR)$ of the cochain complex $(C_{\grp}^\bullet(G ; \AR), \delta)$ is called the \emph{\textup{(}ordinary\textup{)} group cohomology of $G$}.

It is known that the cohomology of group $G$ is canonically isomorphic to the cohomology of classifying space $BG^{\delta}$ of discrete group $G^{\delta}$.
This isomorphism is given by an isomorphism of cochains (see, for example, \cite{D}).
Under this isomorphism, we identify $H^\bullet(BG^{\delta};\AR)$ with $H_{\grp}^\bullet(G;\AR)$.

Let $\AR = \ZZ$ or $\RR$.
Let $C_b^n(G;\AR)$ denote the set of bounded $n$-cochains, i.e., $c \in C_{\grp}^n(G;\AR)$ such that
\[\|c\|_{\infty} = \sup_{g_1,\dots,g_n \in G}|c(g_1,\dots,g_n)| < +\infty  .\]
The cohomology $H_b^\bullet(G,\AR)$ of the cochain complex $(C_b^\bullet(G ; \AR), \delta)$ is called the \emph{bounded cohomology of $G$}.
The inclusion map from $C_b^\bullet(G;\AR)$ to $C_{\grp}^\bullet(G;\AR)$ induces the homomorphism $c_G \colon H_b^\bullet(G;\AR) \to H_{\grp}^\bullet(G;\AR)$, which is called the \textit{comparison map}.

\begin{definition}
A real-valued function $\mu$ on a group $G$ is called a \emph{quasi-morphism} if
\[D(\mu) = \sup_{g,h \in G}|\mu(gh)-\mu(g)-\mu(h)|\]
is finite.
The value $D(\mu)$ is called the {\it defect} of $\mu$.
A quasi-morphism $\mu$ on $G$ is called \emph{homogeneous} if $\mu(g^n)=n\mu(g)$ for all $g \in G$ and $n \in \ZZ$.
Let $Q(G)$ denote the real vector space of homogeneous quasi-morphisms on $G$.
\end{definition}

It is known that any homogeneous quasi-morphism is conjugation-invariant, that is, $\mu \in Q(G)$ satisfies
\begin{align}\label{conj_inv}
  \mu(ghg^{-1}) = \mu(h)
\end{align}
for any $g, h \in G$ (see \cite[Section 2.2.3]{Cal} for example).

By definition, the coboundary $\delta \mu$ of a homogeneous quasi-morphism $\mu \in Q(G)$ defines a bounded two-cocycle on $G$. 
This induces the following exact sequence
\begin{align}\label{ex seq qm bdd coh}
  0 \to H_{\grp}^1(G;\RR) \to Q(G) \xrightarrow{\mathbf{d}} H_b^2(G;\RR) \xrightarrow{c_G} H_{\grp}^2(G;\RR)
\end{align}
(see \cite[Theorem 2.50]{Cal} for example).

The following property of homogeneous quasi-morphisms is important in the present paper:
\begin{equation}\label{abel tashizan}
\mu(fg)=\mu(f)+\mu(g)=\mu(gf) \text{ for any }f,g\in G\text{ with }fg=gf
\end{equation}
(see \cite[Proposition 3.1.4]{PR} for example).



In the present paper, we often refer to Ostrover's Calabi quasi-morphism and so we explain here.
Let $(S^2 \times S^2, \omega_{\lambda})$ be the symplectic manifold defined in Subsection \ref{sec boundedness of ch class}.
Entov and Polterovich \cite{EP03} constructed a homogeneous quasi-morphism $\mu^1$ on $\widetilde{\Ham}(S^2\times S^2,\omega_1)$ using the Hamiltonian Floer theory.
More precisely, $\mu^1$ is constructed as the homogenization of Oh-Schwarz's spectral invariants, which is a Hamiltonian Floer theoretic invariant \cite{Sch}, \cite{Oh05}.
(See also Remark \ref{spectral invariant}.)
They also proved that $\mu^1$ descends to $\Ham(S^2\times S^2,\omega_1)$.

After their work, Ostrover \cite{Os} applied Entov-Polterovich's idea to $\widetilde{\Ham}(S^2\times S^2,\omega_\lambda)$ for $\lambda>1$ and studied a quasi-morphism $\mu^\lambda\colon \widetilde{\Ham}(S^2\times S^2,\omega_\lambda) \to \RR$.
In contrast to Entov-Polterovich's quasi-morphisms, Ostrover's Calabi quasi-morphism $\mu^\lambda$ does not descend to $\Ham(S^2\times S^2,\omega_\lambda)$.

\begin{proposition}[{\cite{Os}}]\label{ostrover loop}
For $\lambda>1$, there exists $\tg \in \pi_1(\Ham(S^2\times S^2,\omega_\lambda))$ such that $\mu(\tg) \neq 0$.
In particular, $\mu^\lambda$ does not descend to $\Ham(S^2\times S^2,\omega_\lambda)$.
\end{proposition}

\subsection{Characteristic classes}
For a fibration, the \textit{primary obstruction class} is defined as an obstruction to the construction of a cross-section.
We briefly recall the definition of the obstruction class via the Serre spectral sequence (see \cite{W} for details).
Let $F \to E \to B$ be a fibration.
For simplicity, we suppose the following; the base space $B$ is one-connected, the fiber $F$ is path-connected, and the fundamental group $\pi_1(F)$ is abelian.
Let $(E_r^{p,q}, d_r^{p,q})$ be the Serre spectral sequence with coefficients in $\pi_1(F)$.
Since $B$ is one-connected, any local coefficient system on $B$ is simple, and therefore we have
\[
  E_2^{p,q} \cong H^p(B;H^q(F;\pi_1(F))).
\]
Hence we obtain $E_2^{2,0} \cong H^2(B;\pi_1(F))$ and $E_2^{0,1} \cong H^1(F;\pi_1(F))$.
Since the cohomology group $H^1(F;\pi_1(F))$ is isomorphic to 
$\Hom(\pi_1(F);\pi_1(F))$, 
the derivation map $d_2^{0,1} \colon E_2^{0,1} \to E_2^{2,0}$ defines a map
\[
  d_2^{0,1} \colon \Hom(\pi_1(F),\pi_1(F)) \to H^2(B;\pi_1(F)).
\]
Here we abuse the symbol $d_2^{0,1}$.

We are now ready to state the definition of the primary obstruction class of fibrations.
\begin{definition}\label{def:obs_class_via_s.s.}
  Let $F \to E \to B$ be a fibration such that $B$ is one-connected, $F$ is path-connected, and $\pi_1(F)$ is abelian.
  Let $(E_r^{p,q}, d_r^{p,q})$ be the Serre spectral sequence of the fibration.
  The cohomology class $\mathfrak{o}(E) = -d_2^{0,1}(\id_{\pi_1(F)}) \in H^2(B;\pi_1(F))$ is called the \textit{primary obstruction class} of $E$, where $\id_{\pi_1(F)}\in \Hom(\pi_1(F),\pi_1(F))$ is the identity homomorphism.
\end{definition}

\begin{remark}\label{rem:obs_section}
  It is known that the above definition is equivalent to the classical definition of the obstruction class to the construction of a cross-section (see, for example, \cite[(6.10) Corollary in Chapter VI and (7.9*) Theorem in Chapter XIII]{W}).
\end{remark}

By the naturality of the spectral sequence, the primary obstruction class is a characteristic class.
Its universal element $\mathfrak{o}$ is given as the primary obstruction class of the principal universal bundle $G \to EG \to BG$.
Note that the classifying space $BG$ is one-connected and $\pi_1(G)$ is abelian.


\begin{remark}
  The class $\mathfrak{o}$ is also obtained as follows.
  By taking classifying spaces of the central extension $0 \to \pi_1(G) \to \tG \to G \to 1$, we obtain the following fibration
  \begin{align}\label{seq:classifying_spaces_univ_cover_top}
    B\pi_1(G) \to B\tG \to BG.
  \end{align}
  Note that the fundamental group of $B\pi_1(G)$ is isomorphic to $\pi_1(G)$ and this is abelian.
  Then, the primary obstruction class of fibration (\ref{seq:classifying_spaces_univ_cover_top}) is the class $\mathfrak{o} \in H^2(BG;\pi_1(G))$.
\end{remark}

Let $f \colon \pi_1(G) \to \RR$ be a homomorphism and
\[
  f_* \colon H^{\bullet}(-;\pi_1(G)) \to H^{\bullet}(-;\RR)
\]
denote the change of coefficients homomorphism.
Let $(E_r^{p,q}, d_r^{p,q})$ be the Serre spectral sequence of (\ref{seq:classifying_spaces_univ_cover_top}) with coefficients in $\RR$.
Since $E_2^{0,1} \cong H^1(B\pi_1(G);\RR) \cong \Hom(\pi_1(G);\RR)$ and $E_2^{2,0} \cong H^2(BG;\RR)$, the derivation $d_2^{0,1} \colon E_2^{0,1} \to E_2^{2,0}$ defines a homomorphism
\[
  d_2^{0,1} \colon \Hom(\pi_1(G), \RR) \to H^2(BG;\RR).
\]

\begin{proposition}\label{prop:transgression_obstruction_homom}
  Let $(E_r^{p,q}, d_r^{p,q})$ be the Serre spectral sequence of $(\ref{seq:classifying_spaces_univ_cover_top})$ with coefficients in $\RR$.
  For a homomorphism $f \colon \pi_1(G) \to \RR$, the equality
  \[
    -d_2^{0,1}(f) = f_* \mathfrak{o} \in H^2(BG;\RR)
  \]
  holds.
\end{proposition}

\begin{proof}
  Let $(E_r^{'p,q}, d_r^{'p,q})$ be the Serre spectral sequence of (\ref{seq:classifying_spaces_univ_cover_top}) with coefficients in $\pi_1(G)$.
  Then the equality $-d_2^{'0,1}(\id_{\pi_1(G)}) = \mathfrak{o}$ holds.
  Since the derivation maps in the Serre spectral sequence is compatible with the change of coefficients homomorphisms, we have the following commutative diagram
  \[
  \xymatrix{
  \Hom(\pi_1(G), \pi_1(G)) \cong E_2^{'0,1} \ar[r]^-{d_2^{'0,1}} \ar[d]^-{f_*} & E_2^{'2,0} \cong H^2(BG;\pi_1(G)) \ar[d]^-{f_*} \\
  \Hom(\pi_1(G), \RR) \cong E_2^{0,1} \ar[r]^-{d_2^{0,1}} & E_2^{2,0} \cong H^2(BG;\RR).
  }
  \]
  Since $f = f_*(\id_{\pi_1(G)})$, we obtain
  \[
    -d_2^{0,1}(f) = -d_2^{0,1}(f_*(\id_{\pi_1(G)})) = f_*(-d_2^{'0,1}(\id_{\pi_1(G)})) = f_* \mathfrak{o}
  \]
  and the proposition follows.
\end{proof}

\begin{remark}\label{remark:derivation_is_isom}
  Let $(E_r^{p,q}, d_r^{p,q})$ be the Serre spectral sequence of (\ref{seq:classifying_spaces_univ_cover_top}) with coefficients in $\RR$.
  Then, the map $d_2^{0,1}$ is an isomorphism.
  Indeed, the $E_2$-page of the spectral sequence induces an exact sequence
  \begin{align*}
    0 \to H^1(BG;\RR) \to H^1(&B\tG;\RR) \to H^1(B\pi_1(G);\RR)\\
    &\xrightarrow{d_2^{0,1}} H^2(BG;\RR) \to H^2(B\tG;\RR).
  \end{align*}
  Since $\tG$ is one-connected, the classifying space $B\tG$ is two-connected.
  Hence the cohomology groups $H^1(B\tG;\RR)$ and $H^2(B\tG;\RR)$ are trivial, and this implies that the derivation map $d_2^{0,1}$ is an isomorphism.
  In particular, the class $f_*\mathfrak{o} = -d_2^{0,1}(f)$ is non-zero if and only if the homomorphism $f$ is non-zero.
\end{remark}

\section{Construction of group cohomology classes}\label{cohom class sec}

Let us consider an exact sequence
\begin{align}\label{seq:group_extension}
  1 \to K \xrightarrow{i} \GG \xrightarrow{\pi} G \to 1
\end{align}
of discrete groups.

\begin{definition}
  A subspace $\mathcal{C}(\GG)$ of $C^1(\GG;A)$ is defined by
  \begin{align}\label{CGamma}
    \mathcal{C}(\GG) = \{ F \in & C_{\grp}^1(\GG;A) \nonumber \\
    & \mid F(k\gg) = F(\gg k) = F(\gg) + F(k) \text{ for any } \gg \in \GG, k \in K \}.
  \end{align}
\end{definition}

We define a map $\mathfrak{D} \colon \mathcal{C}(\GG) \to C_{\grp}^2(G;A)$ by setting
\[
  \mathfrak{D}(F)(g_1, g_2) = F(\gg_2) - F(\gg_1 \gg_2) + F(\gg_1),
\]
where $\gg_j$ is an element of $\GG$ satisfying $\pi(\gg_j) = g_j$.

\begin{lemma}\label{lemma:I_welldef}
  The map $\mathfrak{D} \colon \mathcal{C}(\GG) \to C_{\grp}^2(G;A)$ is well-defined.
\end{lemma}

\begin{proof}
  Let $\gg_j'$ be another element of $\GG$ satisfying $\pi(\gg_j') = g_j$.
  Then there exist $k_1, k_2 \in K$ satisfying $\gg_1' = k_1 \gg_1$ and $\gg_2' = \gg_2 k_2$.
  By the definition of $\mathcal{C}(\GG)$, we have
  \begin{align*}
    &F(\gg_2') - F(\gg_1' \gg_2') + F(\gg_1')\\
    &= F(\gg_2 k_2) - F(k_1 \gg_1 \gg_2 k_2) + F(k_1 \gg_1)\\
    &= (F(\gg_2) + F(k_2)) - (F(k_1) + F(\gg_1 \gg_2) + F(k_2)) + (F(k_1) + F(\gg_1))\\
    &= F(\gg_2) - F(\gg_1 \gg_2) + F(\gg_1).
  \end{align*}
  This implies the well-definedness of the map $\mathfrak{D}$.
\end{proof}

\begin{lemma}\label{lemma:I_cocycle}
  For any $F \in \mathcal{C}(\GG)$, the cochain $\mathfrak{D}(F)$ is a cocycle.
\end{lemma}

\begin{proof}
  Since $\pi^* \mathfrak{D}(F) = -\delta F$ by the definition of $\mathfrak{D}(f)$, we have
  \[
    \pi^* (\delta \mathfrak{D}(F)) = -\delta \delta F = 0.
  \]
  By the surjectivity of $\pi \colon \GG \to G$, we have $\delta \mathfrak{D}(F) = 0$.
\end{proof}

\begin{definition}\label{def:map_frak_d}
  A homomorphism $\mathfrak{d} \colon \mathcal{C}(\GG) \to H_{\grp}^2(G;A)$ is defined by
  \[
    \mathfrak{d}(F) = [\mathfrak{D}(F)] \in H_{\grp}^2(G;A).
  \]
\end{definition}

For an element $F$ of $\mathcal{C}(\GG)$, the restriction $F|_{K} = i^*F$ to $K$ is a homomorphism. 
Moreover, $F|_{K}$ is $\GG$-invariant since
\[
  F(\gg^{-1}k\gg) = F(\gg \cdot \gg^{-1} k \gg) - F(\gg) = F(k \gg) - F(\gg) = F(k).
\]
Let $H_{\grp}^1(K;A)^\GG$ denote the space of $\GG$-invariant homomorphisms from $K$ to $A$.
Then the restriction to $K$ defines a homomorphism $i^* \colon \mathcal{C}(\GG) \to H_{\grp}^1(K;A)^\GG$.

\begin{lemma}\label{lemma:surj_C_H^1(K)}
  The homomorphism $i^* \colon \mathcal{C}(\GG) \to H_{\grp}^1(K;A)^\GG$ is surjective.
\end{lemma}

\begin{proof}
  Let $s \colon G \to \GG$ be a section of $p \colon \GG \to G$ satisfying $s(1_G) = 1_{\GG}$, where $1_G \in G$ and $1_{\GG} \in \GG$ are the unit elements of $G$ and $\GG$, respectively.
  Since $\gg \cdot s(\pi(\gg))^{-1}$ is in $\ke(\pi \colon \GG \to G)$, we regard $\gg \cdot s(\pi(\gg))^{-1}$ as an element of $K$ under the injection $i \colon K \to \GG$.
  For an element $f$ of $H_{\grp}^1(K;A)^\GG$, define $f_s \colon \GG \to A$ by
  \[
    f_s(\gg) = f(\gg \cdot s(\pi(\gg))^{-1}).
  \]
  Note that the restriction of $f_s$ to $K$ is equal to $f$.
  Moreover, the equalities
  \begin{align*}
    f_s(k\gg) &= f(k\gg \cdot s(\pi(k\gg))^{-1}) = f(k\gg \cdot s(\pi(\gg))^{-1})\\
    &= f_s(k) + f_s(\gg \cdot s(\pi(k\gg))^{-1}) = f_s(k) + f_s(\gg)
  \end{align*}
  and
  \begin{align*}
    f_s(\gg k) &= f(\gg k \cdot s(\pi(\gg k))^{-1}) = f(\gg k \cdot s(\pi(\gg))^{-1})\\
    &= f(\gg k \gg^{-1}) + f_s(\gg \cdot s(\pi(k\gg))^{-1})\\
    &= f(k) + f_s(\gg \cdot s(\pi(k\gg))^{-1})\\
    &= f_s(k) + f_s(\gg \cdot s(\pi(k\gg))^{-1}) = f_s(k) + f_s(\gg)
  \end{align*}
  hold, where we use the $\GG$-invariance of $f$ in the second equalities.
  Hence $f_s$ is an element of $\mathcal{C}(\GG)$ and the surjectivity follows.
\end{proof}

For sequence (\ref{seq:group_extension}), there is an exact sequence
\begin{align}\label{seq:inf-res_ex_seq}
  0 \to H_{\grp}^1(G;\RR) \xrightarrow{\pi^*} H_{\grp}^1(&\GG;\RR) \xrightarrow{i^*} H_{\grp}^1(K;\RR)^{\GG} \nonumber\\
  &\xrightarrow{\tau} H_{\grp}^2(G;\RR) \xrightarrow{\pi^*} H_{\grp}^2(\GG;\RR)
\end{align}
called the \textit{five-term exact sequence}.
This five-term exact sequence is obtained by the Hochschild-Serre spectral sequence $(E_r^{p,q}, d_r^{p,q})$ of (\ref{seq:group_extension}), and the map $\tau$ is the derivation $d_2^{0,1} \colon E_2^{0,1} = H_{\grp}^1(K;\RR)^{\GG} \to E_2^{2,0} = H_{\grp}^2(G;\RR)$.

\begin{lemma}\label{lem:commut_diag_basics}
  The diagram
  \[
  \xymatrix{
  \mathcal{C}(\GG) \ar[rd]^-{\mathfrak{d}} \ar[d]^-{i^*} & \\
  H_{\grp}^1(K;A)^{\GG} \ar[r]_-{\tau} & H_{\grp}^2(G;A).
  }
  \]
  commutes.
\end{lemma}

\begin{proof}
  By Definition \ref{def:map_frak_d} and Proposition \ref{prop:transgression_cocycle_neukirch} below, the commutativity follows.
\end{proof}

\begin{proposition}[{\cite[(1.6.6) Proposition]{MR2392026}}]\label{prop:transgression_cocycle_neukirch}
  For any $\GG$-invariant homomorphism $f \in H_{\grp}^1(K;A)^{\GG}$, there exists a one-cochain $F \colon \GG \to A$ such that $i^* F = f$ and that $\delta F(\gg_1, \gg_2)$ depends only on $\pi(\gg_1)$ and $\pi(\gg_2)$, that is, there exists a cocycle $c \in C_{\grp}^2(G;A)$ satisfying $c(\pi(\gg_1), \pi(\gg_2)) = \delta F(\gg_1, \gg_2)$ for any $\gg_1, \gg_2 \in \GG$.
  Moreover, the class $\tau(f)$ is equal to $[c] \in H_{\grp}^2(G;A)$.
\end{proposition}

\section{A diagram via bounded cohomology and quasi-morphism}\label{diagram sec}
From this section, we mainly consider cohomology with coefficients in $\RR$.
In this section, we refine the commutative diagram in view of bounded cohomology and homogeneous quasi-morphism.
Recall that a cohomology class $\alpha \in H_{\grp}^2(G;\RR)$ is called \textit{bounded} if $\alpha$ is in the image of the comparison map $c_G \colon H_b^2(G;\RR) \to H_{\grp}^2(G;\RR)$.

\begin{proposition}\label{prop:compatibility_transgression_qm_construction}
  There is a commutative diagram
  \[
  \xymatrix{
  \mathcal{C}(\GG) \cap Q(\GG) \ar[r]^-{\mathfrak{d}_b} \ar[rd]^-{\mathfrak{d}} \ar[d]^-{i^*} & H_b^2(G;\RR) \ar[d]^-{c_G}\\
  H_{\grp}^1(K;\RR)^G \ar[r]_-{\tau} & H_{\grp}^2(G;\RR).
  }
  \]
\end{proposition}

\begin{proof}
  Let $F$ be an element of $\mathcal{C}(\GG) \cap Q(\GG)$.
  Then, the cocycle $\mathfrak{D}(F)$ is bounded since $F$ is a quasi-morphism and $\mathfrak{D}(F)(g_1, g_2) = F(\gg_2) - F(\gg_1\gg_2) + F(\gg_1)$ for any $g_1, g_2 \in G$ and their lifts $\gg_1, \gg_2 \in \GG$.
  Hence the homomorphism $\mathfrak{D} \colon \mathcal{C}(\GG) \to C_{\grp}^2(G;\RR)$ induces a homomorphism
  \[
    \mathfrak{d}_b \colon \mathcal{C}(\GG) \cap Q(\GG) \to H_b^2(G;\RR).
  \]
  By the definition of the comparison map $c_G$, we have $\mathfrak{d} = c_G \circ \mathfrak{d}_b$.
\end{proof}

\begin{remark}\label{remark:central_ext_qm_bdd_comm_diag}
  For a central extension
  \[
    0 \to A \xrightarrow{i} \GG \xrightarrow{\pi} G \to 1,
  \]
  the space $Q(\GG)$ is contained in $\mathcal{C}(\GG)$.
  Indeed, by the definition of central extension, we have $a \gg = \gg a$ for any $a \in A$ and $\gg \in \GG$.
  Hence, by (\ref{abel tashizan}), any homogeneous quasi-morphism $\mu \in Q(\GG)$ satisfies
  \[
    \mu(a \gg) = \mu(\gg a) = \mu(a) + \mu(\gg).
  \]
  This implies that $Q(\GG) \subset \mathcal{C}(\GG)$.
  Moreover, any homomorphism $f \colon A \to \RR$ is $\GG$-invariant since
  $\gg^{-1} a \gg = a \gg^{-1} \gg = a$ for any $\gg \in \GG$ and any $a \in A$.
  Hence, together with Proposition \ref{prop:compatibility_transgression_qm_construction}, we obtain the following commutative diagram
  \[
  \xymatrix{
  Q(\GG) \ar[r]^-{\mathfrak{d}_b} \ar[rd]^-{\mathfrak{d}} \ar[d]^-{i^*} & H_b^2(G;\RR) \ar[d]^-{c_G}\\
  H_{\grp}^1(A;\RR) \ar[r]_-{\tau} & H_{\grp}^2(G;\RR)
  }
  \]
  for a central extension $0 \to A \to \GG \to G \to 1$.
\end{remark}

\begin{lemma}\label{lemma:qm_hom_on_K_CGG}
  Let $\mu$ be a homogeneous quasi-morphism on $\GG$ whose restriction to $K$ is a homomorphism.
  Then $\mu$ is contained in $\mathcal{C}(\GG)$.
\end{lemma}

\begin{proof}
  For any $\gg \in \GG$, $k \in K$, and $n \in \mathbb{N}$, the equalities
  \[
    (k\gg)^n = k \cdot \gg k \gg^{-1} \cdot \gg^2 k \gg^{-2} \cdot \cdots \cdot \gg^{n-1} k \gg^{-(n-1)} \cdot \gg^{n}
  \]
  and
  \[
    (\gg k)^n = \gg^n \cdot \gg^{-(n-1)} k \gg^{n-1} \cdot \cdots \cdot \gg^{-2} k \gg^2 \cdot \gg^{-1} k \gg \cdot k
  \]
  hold.
  By (\ref{conj_inv}), the restriction $\mu|_{K}$ is $\GG$-invariant.
  Hence we have
  \[
    \mu(k \cdot \gg k \gg^{-1} \cdot \gg^2 k \gg^{-2} \cdot \cdots \cdot \gg^{n-1} k \gg^{-(n-1)}) = \mu(k^n)
  \]
  and
  \[
    \mu(\gg^{-(n-1)} k \gg^{n-1} \cdot \cdots \cdot \gg^{-2} k \gg^2 \cdot \gg^{-1} k \gg \cdot k) = \mu(k^n).
  \]
  These equalities imply that
  \begin{align*}
    n\cdot |\mu(k\gg) - \mu(k) - \mu(\gg)| = |\mu((k\gg)^n) - \mu(k^n) - \mu(\gg^n)| < D(\mu)
  \end{align*}
  and
  \begin{align*}
    n\cdot |\mu(\gg k) - \mu(\gg) - \mu(k)| = |\mu((\gg k)^n) - \mu(\gg^n) - \mu(k^n)| < D(\mu).
  \end{align*}
  Hence we obtain $\mu(k\gg) = \mu(k) + \mu(\gg)$ and $\mu(\gg k) = \mu(\gg) + \mu(k)$.
\end{proof}

\begin{theorem}\label{thm:isomorphism_theorem_general}
  The homomorphism $\mathfrak{d} \colon \mathcal{C}(\GG) \to H_{\grp}^2(G;\RR)$ induces an isomorphism
  \[
    (\mathcal{C}(\GG) \cap Q(\GG))/(H_{\grp}^1(\GG;\RR) + \pi^* Q(G)) \to \mathrm{Im}(\tau) \cap \mathrm{Im}(c_G).
  \]
\end{theorem}

\begin{proof}
  Let us consider the following commutative diagram whose rows and columns are exact:
  \[
  \xymatrix{
  H^1(K;\RR)^{\GG} \ar[r] & Q(K)^{\GG} \ar[r] & H_b^2(K;\RR)^{\GG} &  \\
  H^1(\GG;\RR) \ar[u] \ar[r] & Q(\GG) \ar[u] \ar[r]^-{\mathbf{d}} & H_b^2(\GG;\RR) \ar[u] \ar[r] & H^2(\GG;\RR) \\
  H^1(G;\RR) \ar[u] \ar[r] & Q(G) \ar[u]^-{\pi^*} \ar[r] & H_b^2(G;\RR) \ar[u]^-{\pi^*} \ar[r]^-{c_G} & H^2(G;\RR) \ar[u]^-{\pi^*} \\
   & & 0 \ar[u] & H^1(K;\RR)^{\GG}, \ar[u]^-{\tau}
  }
  \]
  where the exactness of the third column was shown in \cite{MR1338286}.
  By the definition of $\mathfrak{d}_b$, we have $\pi^* \mathfrak{d}_b(\mu) = \mathbf{d}(\mu)$ for $\mu \in \mathcal{C}(\GG) \cap Q(\GG)$.
  Hence the map $\pi^* \colon H_b^2(G;\RR) \to H_b^2(\GG;\RR)$ gives an isomorphism
  \[
    \pi^* \colon H_b^2(G;\RR) \xrightarrow{\cong} \mathbf{d}(\mathcal{C}(\GG) \cap Q(\GG)).
  \]
  Then, in this diagram, the map $\mathfrak{d}$ is given as the composite
  \[
    c_G \circ (\pi^*)^{-1} \circ \mathbf{d} \colon \mathcal{C}(\GG) \cap Q(\GG) \to H_{\grp}^2(G;\RR).
  \]

  The equality $\ke(\mathfrak{d}) = H_{\grp}^1(\GG;\RR) + \pi^* Q(G)$ is verified by a diagram chasing argument.
  By Lemma \ref{lemma:qm_hom_on_K_CGG} and a diagram chasing argument, the surjectivity of the map $\mathfrak{d} \colon \mathcal{C}(\GG) \cap Q(\GG) \to \mathrm{Im}(\tau) \cap \mathrm{Im}(c_G)$ also follows.
\end{proof}

\begin{remark}\label{remark:isomorphism_theorem_central_case}
  For a central extension $\GG$ of $G$, the homomorphism $\mathfrak{d} \colon \mathcal{C}(\GG) \to H^2(G;\RR)$ induces an isomorphism
  \[
    Q(\GG)/(H_{\grp}^1(\GG;\RR) + \pi^* Q(G)) \to \mathrm{Im}(\tau) \cap \mathrm{Im}(c_G)
  \]
  since $\mathcal{C}(\GG) \cap Q(\GG) = Q(\GG)$ (see Remark \ref{remark:central_ext_qm_bdd_comm_diag}).
\end{remark}

\section{On topological groups}\label{section top int}
\subsection{General topological groups}

Let $G$ be a topological group and $\pi \colon \tG \to G$ the universal covering.
Since the exact sequence
\begin{align}\label{seq:extension_univ_covering}
  0 \to \pi_1(G) \xrightarrow{i} \tG \xrightarrow{\pi} G \to 1
\end{align}
is a central extension (\cite[Theorem 15]{P})), we obtain the commutative diagram
\begin{align}\label{diagram:delta_tau_cG_commute}
  \xymatrix{
  Q(\tG) \ar[r]^-{\mathfrak{d}_b} \ar[rd]^-{\mathfrak{d}} \ar[d]^-{i^*} & H_b^2(G;\RR) \ar[d]^-{c_G}\\
  H_{\grp}^1(\pi_1(G);\RR) \ar[r]_-{\tau} & H_{\grp}^2(G;\RR)
  }
\end{align}
by Remark \ref{remark:central_ext_qm_bdd_comm_diag}.
Moreover, by Remark \ref{remark:isomorphism_theorem_central_case}, the homomorphism $\mathfrak{d} \colon Q(\tG) \to H_{\grp}^2(G;\RR)$ induces an isomorphism
\begin{align}\label{isom:tau_cap_cG}
  Q(\tG)/(H_{\grp}^1(\tG;\RR) + \pi^*Q(G)) \to \mathrm{Im}(\tau) \cap \mathrm{Im}(c_G).
\end{align}

In this section, 
we clarify the relation between the class $\mathfrak{d}(\mu) \in H_{\grp}^2(G;\RR)$ and the primary obstruction class $\mathfrak{o} \in H^2(BG;\pi_1(G))$.

By taking the classifying spaces of (\ref{seq:extension_univ_covering}), we obtain a commutative diagram of fibrations
\[
\xymatrix{
B\pi_1(G) \ar[r] \ar@{=}[d] & B\tG^{\delta} \ar[r] \ar[d] & BG^{\delta} \ar[d]^{B\iota}\\
B\pi_1(G) \ar[r] & B\tG \ar[r] & BG.
}
\]
In what follows, we regard the pullback $B\iota^* \colon H^{\bullet}(BG;\RR) \to H^{\bullet}(BG^{\delta};\RR)$ as a homomorphism
\[
  B\iota^* \colon H^{\bullet}(BG;\RR) \to H_{\grp}^{\bullet}(G;\RR)
\]
under the isomorphism $H^{\bullet}(BG^{\delta};\RR) \cong H_{\grp}^{\bullet}(G;\RR)$.

\begin{lemma}\label{lemma:tau_equal_Bid_2}
  Let $(E_r^{p,q}, d_r^{p,q})$ be the $\RR$-coefficients cohomology Serre spectral sequence of the fibration $B\pi_1(G) \to B\tG \to BG$.
  Then the equality
  \[
    B\iota^* \circ d_2^{0,1} = \tau \colon H_{\grp}^1(\pi_1(G);\RR) \to H_{\grp}^2(G;\RR)
  \]
  holds, where we identify $E_2^{0,1} = H^1(B\pi_1(G);\RR)$ with $H_{\grp}^1(\pi_1(G);\RR)$.
\end{lemma}

\begin{proof}
  Let $({}^{\delta}E_r^{p, q}, {}^{\delta}d_r^{p,q})$ be the Hochschild-Serre spectral sequence of central extension (\ref{seq:extension_univ_covering}).
  Note that the spectral sequence $({}^{\delta}E_r^{p, q}, {}^{\delta}d_r^{p,q})$ is isomorphic to the Serre spectral sequence of the fibration $B\pi_1(G) \to B\tG^{\delta} \to BG^{\delta}$ (see \cite{benson_1991} for example).
  Since the map $\tau$ is equal to the derivation map ${}^{\delta}d_2^{0,1}$ by definition, the naturality of the Serre spectral sequence asserts that
  \[
    B\iota^* \circ d_2^{0,1}= {}^{\delta}d_2^{0,1} = \tau,
  \]
  and the lemma follows.
\end{proof}

\begin{corollary}\label{cor:d_mu_obstruction_class}
  Let $\mathfrak{o} \in H^2(BG;\RR)$ be the primary obstruction class for $G$-bundles.
  Then, for any homogeneous quasi-morphism $\mu \in Q(\tG)$, the equality
  \[
    \mathfrak{d}(\mu) = -B\iota^* ((\mu|_{\pi_1(G)})_* \mathfrak{o})
  \]
  holds.
\end{corollary}

\begin{proof}
  Let $(E_r^{p,q}, d_r^{p,q})$ be the Serre spectral sequence as in Lemma \ref{lemma:tau_equal_Bid_2}.
  Using Proposition \ref{prop:transgression_obstruction_homom}, we obtain
  \[
    B\iota^* \circ d_2^{0,1}(\mu|_{\pi_1(G)}) = -B\iota^* ((\mu|_{\pi_1(G)})_* \mathfrak{o}).
  \]
  On the other hand, using Lemma \ref{lemma:tau_equal_Bid_2} and commutative diagram (\ref{diagram:delta_tau_cG_commute}), we obtain
  \[
    B\iota^* \circ d_2^{0,1}(\mu|_{\pi_1(G)}) = \tau(\mu|_{\pi_1(G)}) = \tau(i^*(\mu)) = \mathfrak{d}(\mu).
  \]
  Hence the equality $\mathfrak{d}(\mu) = -B\iota^* ((\mu|_{\pi_1(G)})_* \mathfrak{o})$ holds.
\end{proof}

\begin{corollary}\label{cor:Biota_injective}
  If $H^1(\tG;\RR)$ is trivial, then the homomorphism
  \[
    B\iota^* \colon H^2(BG;\RR) \to H_{\grp}^2(G;\RR)
  \]
  is injective.
\end{corollary}

\begin{proof}
  By the five-term exact sequence 
  \begin{align*}
    0 \to H_{\grp}^1(G;\RR) \to H_{\grp}^1(&\tG;\RR) \to H_{\grp}^1(\pi_1(G);\RR)\\
    &\xrightarrow{\tau} H_{\grp}^2(G;\RR) \to H_{\grp}^2(\tG;\RR),
  \end{align*}
  the triviality of $H_{\grp}^1(\tG;\RR)$ implies the injectivity of the map $\tau$.
  Hence the map $B\iota^*$ is injective by Lemma \ref{lemma:tau_equal_Bid_2} and Remark \ref{remark:derivation_is_isom}.
\end{proof}

\begin{theorem}\label{thm:isom_thm_qm_bdd_ch_class}
  The homomorphism $\mathfrak{d} \colon Q(\tG) \to H_{\grp}^2(G;\RR)$ induces an isomorphism
  \[
    Q(\tG)/(H_{\grp}^1(\tG;\RR) + \pi^*Q(G)) \to \mathrm{Im}(B\iota^*) \cap \mathrm{Im}(c_G).
  \]
\end{theorem}

\begin{proof}
  The equality
  \[
    \mathrm{Im}(B\iota^*) = \mathrm{Im}(\tau)
  \]
  holds by Lemma \ref{lemma:tau_equal_Bid_2} and Remark \ref{remark:derivation_is_isom}.
  Hence, isomorphism (\ref{isom:tau_cap_cG}) implies the theorem.
\end{proof}

The following corollary immediately follows from Theorem \ref{thm:isom_thm_qm_bdd_ch_class}.

\begin{corollary}\label{cor:isom_thm_tG_perfect_case}
  If the first cohomology $H_{\grp}^1(\tG;\RR)$ is trivial, then the homomorphism $\mathfrak{d}$ induces an isomorphism
  \[
    Q(\tG)/\pi^*Q(G) \to \mathrm{Im}(B\iota^*) \cap \mathrm{Im}(c_G).
  \]
  In particular, if $\mu \in Q(\tG)$ does not descend to $G$, then the class $\mathfrak{d}(\mu) \in H_{\grp}^2(G;\RR)$ is non-zero.
\end{corollary}

By using Corollary \ref{cor:d_mu_obstruction_class}, Corollary \ref{cor:Biota_injective}, and Theorem \ref{thm:isom_thm_qm_bdd_ch_class}, we obtain the following corollary.

\begin{corollary}\label{cor:bounded_unbounded_qm}
Let $G$ be a topological group and $\tG$ the universal covering of $G$.
  \begin{enumerate}
    \item
    Let $\tmu\colon \tG \to \RR$ be a homogeneous quasi-morphism which does not descend to $G$.
    Let $\mathfrak{o} \in H^2(BG;\pi_1(G))$ denote the primary obstruction class of $G$.
    Then, the cohomology class
    \[
      B\iota^*(((\mu|_{\pi_1(G)})_*\mathfrak{o})_{\RR}) \in H_{\grp}^2(G;\RR)
    \]
    is bounded.
    Here, $(\mu|_{\pi_1(G)})_* \colon H^2(BG;\pi_1(G)) \to H^2(BG;\RR)$ is the change of coefficients homomorphism induced from $\mu|_{\pi_1(G)} \colon \pi_1(G) \to \RR$.

    \item
    Assume that the space $Q(\tG)$ is trivial.
    Then, for any non-zero element $\mathfrak{c}$ of $H^2(BG;\RR)$, a cohomology class
    \[
      (B\iota)^*(\mathfrak{c}) \in H_{\grp}^2(G;\RR)
    \]
    is unbounded.
  \end{enumerate}
\end{corollary}

\subsection{Hamiltonian and contact Hamiltonian diffeomorphism groups}

We set $G_{\lambda} = \Ham(S^2 \times S^2, \omega_{\lambda})$ and $H = \Cont_0(S^3,\xi)$.
For $1 < \lambda \leq 2$, it is known that $\pi_1(G_{\lambda}) \cong \ZZ \times \ZZ/2\ZZ \times \ZZ/2\ZZ$ (\cite{A}) and $\pi_1(H) \cong \ZZ$ (\cite{El}).
By Remark \ref{remark:derivation_is_isom}, we have
\[
  H^2(BG_{\lambda};\ZZ) \cong H^1(B\pi_1(G_{\lambda});\ZZ) \cong \Hom(\pi_1(G_{\lambda}),\ZZ) \cong \ZZ
\]
and
\[
  H^2(BH;\ZZ) \cong H^1(B\pi_1(H);\ZZ) \cong \Hom(\pi_1(H),\ZZ) \cong \ZZ.
\]
Let $\mathfrak{o}_{H}$ be the primary obstruction class of $H$-bundles, which is a generator of $H^2(BH;\ZZ)$.
The primary obstruction class $\mathfrak{o}$ of $G_{\lambda}$-bundles is defined as an identity homomorphism in $H^2(BG_{\lambda};\pi_1(G_{\lambda})) \cong \Hom(\pi_1(G_{\lambda}),\pi_1(G_{\lambda}))$.
Let
\begin{align}\label{homom_phi}
  \phi \colon \pi_1(G_{\lambda}) \cong \ZZ \times \ZZ/2\ZZ \times \ZZ/2\ZZ \to \ZZ
\end{align}
be the homomorphism sending $(n, a, b) \in \ZZ \times \ZZ/2\ZZ \times \ZZ/2\ZZ$ to $n \in \ZZ$.
We set
\[
  \mathfrak{o}_{G_{\lambda}} = \phi_*\mathfrak{o} \in H^2(BG_{\lambda};\ZZ).
\]
Then the class $\mathfrak{o}_{G_{\lambda}}$ is a generator of $H^2(BG_{\lambda};\ZZ)$.


\begin{proof}[Proof of Corollary $\ref{cor bounded and unbounded}$]

 First, we prove (1).
 Recall that the restriction $\mu^{\lambda}|_{\pi_1(G_{\lambda})}$ of Ostrover's Calabi quasi-morphism is a non-trivial homomorphism to $\RR$ (Proposition \ref{ostrover loop}). Hence there exists a non-zero constant $a$ such that
  \[
    \phi = a\mu^{\lambda}|_{\pi_1(G_{\lambda})} \colon \pi_1(G_{\lambda}) \to \RR,
  \]
  where $\phi$ is the homomorphism given as (\ref{homom_phi}).
  Therefore we have
  \[
    B\iota^*(\mathfrak{o}_{G_{\lambda}})_{\RR} = a\cdot B\iota^*((\mu|_{\pi_1(G_{\lambda})})_*\mathfrak{o})_{\RR}.
  \]
  Since Ostrover's Calabi quasi-morphism does not descend to $G_\lambda$, the class $B\iota^*(\mathfrak{o}_{G_{\lambda}})_{\RR}$ is bounded by Corollary \ref{cor:bounded_unbounded_qm} (1).

  Next, we prove (2).
  Because the universal covering group $\widetilde{H} = \widetilde{\Cont}_0(S^3,\xi)$ is uniformly perfect (see \cite[Corollary 3.6 and Remark 3.7]{FPR18}), we have $Q(\widetilde{H}) = 0$.
  By Corollary \ref{cor:bounded_unbounded_qm} (2), the class $B\iota^*(\mathfrak{o}_H)_{\RR}$ is unbounded.
\end{proof}

In Section \ref{MW ineq}, we provide another proof of Corollary \ref{cor bounded and unbounded} (2)  by using a Milnor-Wood type inequality (Theorem \ref{MW inequality}) instead of Corollary \ref{cor:bounded_unbounded_qm} (2) (see Remark \ref{large cn}).

We end this section with a proof of Corollary \ref{bdd cohom on Ham}.
To do this, we prepare the following lemma.

\begin{lemma}\label{lem inj to bdd coh}
  For any topological group $G$ whose universal covering group $\tG$ satisfies $H_{\grp}^1(\tG;\RR) = 0$, the map
  \[
    \mathfrak{d}_b \colon Q(\tG) \to H_b^2(G;\RR)
  \]
  is injective.
\end{lemma}

\begin{proof}
  By exact sequence (\ref{ex seq qm bdd coh}) and the assumption $H_{\grp}^1(\tG;\RR) = 0$, the map $Q(\tG) \to H_b^2(\tG;\RR)$ is injective. Hence, for any homogeneous quasi-morphism $\tmu \in Q(\tG)$, the bounded cohomology class $[\delta \tmu] \in H_b^2(\tG;\RR)$ is non-zero.
  Since
  \[
    \pi^*(\mathfrak{d}_b(\tmu)) = [\delta \tmu] \in H_b^2(\tG;\RR),
  \]
  where $\pi^* \colon H_b^2(G;\RR) \to H_b^2(\tG;\RR)$ is the homomorphism induced by the universal covering $\pi \colon \tG \to G$, the class $\mathfrak{d}_b(\tmu)$ is also non-zero.
\end{proof}

\begin{proof}[Proof of Corollary $\ref{bdd cohom on Ham}$]
  Because $\widetilde{\Ham}(M,\omega)$ is perfect (\cite{Ba78}), Lemma \ref{lem inj to bdd coh} implies Corollary \ref{bdd cohom on Ham}.
\end{proof}

\section{Milnor-Wood type inequality and bundles with no flat structures}\label{MW ineq}

In this section, we show the existence of bundles over a surface which do not admit foliated (flat) structures.
To this end, first we introduce a Milnor-Wood type inequality.

Let $c$ be a universal characteristic class of foliated principal $G$-bundles.
Then the characteristic class $c$ is given as an element in $H^2(BG^{\delta};\RR)$.
For a foliated principal $G$-bundle $G \to E \to B$, the characteristic class $c(E)$ of $E$ associated to $c$ is defined by
\[
  c(E) = f^*c \in H^2(B;\RR),
\]
where $f \colon B \to BG^{\delta}$ is the classifying map of $E$.

Let $\Sigma_h$ denote a closed oriented surface of genus $h \geq 1$ and $G \to E \to \Sigma_h$ be a foliated $G$-bundle.
Let $\rho \colon \pi_1(\Sigma_h) \to G$ be a holonomy homomorphism of the bundle $E$.
Then the classifying map of the bundle $E$ is given by
\[
  B\rho \colon \Sigma_h \simeq B\pi_1(\Sigma_h) \to BG^{\delta}.
\]

\begin{theorem}\label{MW inequality}
  Let $c$ be an element of $\mathrm{Im}(c_G) \cap \mathrm{Im}(B\iota^*)$ and $[\Sigma_h] \in H_2(\Sigma_h;\ZZ)$ the fundamental class of $\Sigma_h$.
  Then, for any foliated principal $G$-bundle $G \to E \to \Sigma_h$,
  an inequality
  \[
    |\langle c(E), [\Sigma_h] \rangle | \leq D(\tmu)(4h-4)
  \]
  holds, where $\tmu \in Q(\tG)$ is a homogeneous quasi-morphism satisfying $\mathfrak{d}(\tmu) = c$.
\end{theorem}

\begin{proof}
  By Theorem \ref{thm:isom_thm_qm_bdd_ch_class}, there exists a homogeneous quasi-morphism $\tmu \in Q(\tG)$ satisfying $\mathfrak{d}(\tmu) = [\mathfrak{D}(\tmu)] = c$.
  Since $\pi^* \mathfrak{D}(\tmu) = \delta \tmu$, we have
  \[
    \| \mathfrak{D}(\tmu) \|_{\infty} = \| \delta \tmu \|_{\infty} = D(\tmu).
  \]
  In particular, we have $\| \mathfrak{d}(\tmu) \|_{\infty} \leq D(\tmu)$.
  Let $\rho \colon \pi_1(\Sigma_h) \to G$ be a holonomy homomorphism associated with the foliated bundle $G \to E \to \Sigma_h$.
  Since the operator norm of $\rho^* \colon H_{b}^2(G;\RR) \to H_{b}^2(\pi_1(\Sigma_h);\RR)$ is equal or lower than $1$, we have
  \[
    \|\rho^*(\mathfrak{d}(\tmu))\|_{\infty} \leq \|\mathfrak{d}(\tmu)\|_{\infty} \leq D(\tmu).
  \]
  Note that the bounded cohomology of a topological space $X$ is isometrically isomorphic to the bounded cohomology of the fundamental group $\pi_1(X)$ \cite{Gr82}. 
  Hence we have
  \[
    \| c(E) \|_{\infty} = \|\rho^*(\mathfrak{d}(\tmu))\|_{\infty} \leq D(\tmu).
  \]
  Let $\| \Sigma_h \|$ denote the simplicial volume of $\Sigma_h$.
  Then we have $|\langle c, [\Sigma_h] \rangle | \leq \| c \|_{\infty} \| \Sigma_h \|$ (\cite[Proposition 7.10]{Fr}).
  Finally we obtain the inequality
  \[
    |\langle c(E), [\Sigma_h] \rangle |
    \leq \| c \|_{\infty} \| \Sigma_h \|
    \leq D(\tmu) (4h-4).
  \]
\end{proof}

\begin{theorem}\label{general nonflat}
  Let $G$ be a topological group and $\Sigma_h$ a closed surface of genus $h\geq 1$.
  Assume that there exist a homogeneous quasi-morphism $\mu \in Q(\tG)$ and $\gamma \in \pi_1(G)$ satisfying $\mu(\gamma) \neq 0$.
  Then, there exist infinitely many isomorphism classes of principal $G$-bundles over $\Sigma_h$ which do not admit foliated $G$-bundle structures.
\end{theorem}

\begin{proof}
  We normalize the homogeneous quasi-morphism $\tmu$ as $\tmu(\gamma) = 1$ by a non-zero constant multiple.
  We set $c = \mathfrak{d}(\mu) = B\iota^*((\mu|_{\pi_1(G)})_*\mathfrak{o}) \in H^2(BG^\delta;\RR)$, then $c$ belongs to $\mathrm{Im}(c_G) \cap \mathrm{Im}(B\iota^*)$.
  Assume that a principal $G$-bundle $E \to \Sigma_h$ admits a foliated structure.
  Then, there exists a continuous map $f_\delta \colon \Sigma_h \to BG^\delta$ such that $f=B\iota \circ f_\delta$, where $f \colon \Sigma_h \to BG$ is the classifying map of $E$.
  Let $E_\delta$ be a foliated $G$-bundle on $\Sigma_h$ induced from $f^\delta$.
  Then,
  \[
  c(E_\delta) = f_\delta^\ast c = f_\delta^\ast(B\iota^\ast\mu_\ast \mathfrak{o}) = \mu_\ast (f_\delta^\ast B\iota^\ast\mathfrak{o}) = \mu_\ast (f^\ast\mathfrak{o}) = \mu_\ast\mathfrak{o}(E).
  \]

  Hence we obtain that
  \begin{equation}\label{MW cor}
    \langle (\mu|_{\pi_1(G)})_*\mathfrak{o}(E), [\Sigma_h] \rangle = \langle c(E_\delta), [\Sigma_h] \rangle \leq D(\mu)(4h-4)
  \end{equation}
  by Theorem \ref{MW inequality}.

  For each $n \in \ZZ$, we now construct a principal $G$-bundle $E_n$ over $\Sigma_h$ whose characteristic number $\langle (\mu|_{\pi_1(G)})_*\mathfrak{o}(E_n), [\Sigma_h] \rangle = n$.
  Let us fix a triangulation $\mathcal{T}$ of $\Sigma_h$ and take a triangle $\Delta \in \mathcal{T}$.
  For $n \in \ZZ$, we take a loop $\{ g_t \}_{0 \leq t \leq 1}$ in $G$ which represents $\gamma^n \in \pi_1(G)$.
  Let $E \to \Sigma_h \setminus \Int(\Delta)$ and $E' \to \Delta$ be trivial $G$-bundles, where $\Int(\Delta)$ be the interior of $\Delta$.
  Then, we obtain a bundle $E_n$ by gluing the bundles $E$ and $E'$ along $\partial \Delta \approx S^1$ with the transition function $S^1 \to G ; t \to g_t$.
  Since the class $\mathfrak{o}(E_n)$ is the primary obstruction to the cross-sections (see Remark \ref{rem:obs_section}), we have $\langle \mathfrak{o}(E_n), [\Sigma_h] \rangle = \gamma^n$ and therefore we obtain
  \[
    \langle (\mu|_{\pi_1(G)})_*\mathfrak{o}(E_n), [\Sigma_h] \rangle = \mu(\gamma^n) = n.
  \]
  Hence, by equation (\ref{MW cor}), for a sufficiently large $n$, $E_n$ do not admit foliated $G$-bundle structures and we complete the proof.
\end{proof}

\begin{remark}
  Any non-trivial principal $G$-bundle over the $2$-sphere $\Sigma_0$ does not admit foliated structures since the fundamental group of $\Sigma_0$ is trivial.
  Hence, if the order of $\pi_1(G)$ is infinite, there exist infinitely many isomorphism classes of principal $G$-bundles over $\Sigma_0$ which do not admit foliated structures.
\end{remark}

\begin{proof}[Proof of Corollary $\ref{nonflat}$]
  Ostrover's Calabi quasi-morphism satisfies the assumption in Theorem \ref{general nonflat}.
  Hence Theorem \ref{general nonflat} implies the corollary.
\end{proof}

\begin{remark}\label{large cn}
As an application of Theorem \ref{MW inequality}, one can prove Corollary \ref{cor bounded and unbounded} (2) by constructing explicitly a foliated $H$-bundle with arbitrary large characteristic number.
  Indeed, for any $N \in \ZZ = \pi_1(H)$, there exist $2k$ elements $\tg_1, \dots, \tg_{2k}$ of $\widetilde{H}$ such that the equality $N = [\tg_1, \tg_2] \dots [\tg_{2k-1}, \tg_{2k}]$ since $\widetilde{H}$ is uniformly perfect.
  Note that the number $k$ does not depend on $N$.
  We set $g_j = p(\tg_j) \in H$ for any $j$, where $p \colon \widetilde{H} \to H$ is the universal covering.
  Let $\Sigma_k$ be a closed surface of genus $k$ and $a_j \in \pi_1(\Sigma_k)$ the canonical generator with the relation
  \[
    [a_1, a_2] \dots [a_{2k - 1}, a_{2k}] = 1.
  \]
  Let $\varphi \colon \pi_1(\Sigma_k) \to \Cont_0(S^3, \xi)$ be a homomorphism defined by $\varphi(a_j) = g_j$ for any $j$.
  Then, the characteristic number of the foliated $H$-bundle with the holonomy homomorphism $\varphi$ is equal to $N$ (this computation of the characteristic number is known as Milnor's algorithm \cite{Mi58}).
\end{remark}

\section{Non-extendability of homomorphisms on $\pi_1(G)$ to homogeneous quasi-morphisms on $\tG$}\label{nonex sec}

In Section \ref{sec boundedness of ch class}, we use the homogeneous quasi-morphisms on the universal covering $\tG$ to show the (un)boundedness of characteristic classes.
In this section, on the contrary, we use the (un)boundedness of characteristic classes to study the extension problem of homomorphism on $\pi_1(G)$ to $\tG$.
The extension problem of homomorphisms and homogeneous quasi-morphisms have been studied by some researchers (for example, see \cite{I}, \cite{Sht}, \cite{KK}, \cite{KKMM20}, \cite{KKMM21}, \cite{MR4330215}).

Let $T = S^1 \times S^1$ be the two-dimensional torus and $\Homeo_0(T)$ the identity component of the homeomorphism group of $T$ with respect to the compact-open topology.
In \cite{Hams}, it was shown that the fundamental group $\pi_1(\Homeo_0(T))$ is isomorphic to $\ZZ \times \ZZ$.

\begin{corollary}
  Any non-trivial homomorphism in $\Hom(\pi_1(\Homeo_0(T)),\RR)$ cannot be extended to $\widetilde{\Homeo}_0(T)$ as a homogeneous quasi-morphism.
\end{corollary}

\begin{proof}
  It is enough to show that the equality
  \[
    Q(\widetilde{\Homeo}_0(T)) = \pi^*Q(\Homeo_0(T))
  \]
  holds, where $\pi \colon \widetilde{\Homeo}_0(T) \to \Homeo_0(T)$ is the universal covering.
  Because the universal covering $\widetilde{\Homeo}_0(T)$ is perfect \cite{KR}, we have
  \[
    Q(\widetilde{\Homeo}_0(T))/\pi^*Q(\Homeo_0(T)) = \mathrm{Im}(c_G) \cap \mathrm{Im}(B\iota^*)
  \]
  by Corollary \ref{cor main}.
 Because any non-zero classes in $\mathrm{Im}(B\iota^*)$ are unbounded \cite{MR}, we have $Q(\widetilde{\Homeo}_0(T))/\pi^*Q(\Homeo_0(T)) = 0$, and the corollary holds.
\end{proof}

\section*{Acknowledgments}

The authors would like to thank Mitsuaki Kimura, Kevin Li, Yoshifumi Matsuda, Takahiro Matsushita, Masato Mimura, Yoshihiko Mitsumatsu and Kaoru Ono for some comments.

The first author is supported in part by JSPS KAKENHI Grant Number JP18J00765 and 21K13790.
The second author is supported by JSPS KAKENHI Grant Number JP21J11199.

\appendix

\section{Examples of (contact) Hamiltonian fibrations}\label{fibration sec}

Recall that $G_{\lambda} = \Ham(S^2 \times S^2, \omega_{\lambda})$, $H = \Cont_0(S^3,\xi)$, and the cohomology classes
\[
  \mathfrak{o}_{G_{\lambda}} \in H^2(BG_{\lambda};\ZZ) \  \text{ and } \  \mathfrak{o}_{H} \in H^2(BH; \ZZ)
\]
are the primary obstruction classes.
Our main concern in this paper (e.g., Corollary \ref{cor bounded and unbounded}) was the classes $B\iota^*(\mathfrak{o}_{G_{\lambda}})_{\RR} \in H^2(BG_{\lambda}^{\delta};\RR)$ and $B\iota^*(\mathfrak{o}_{H})_{\RR} \in H^2(BH^{\delta}; \RR)$.
In this appendix, we rather use the classes $\mathfrak{o}_{G_{\lambda}}$ and $\mathfrak{o}_{H}$ to study (not necessarily foliated) Hamiltonian fibrations and contact Hamiltonian fibrations.

We bigin with the following genaral proposition.

\begin{proposition}\label{prop G,K bundle difference}
  Let $G$ and $K$ be topological groups and $i \colon G \to K$ be a continuous homomorphism.
  Assume that the universal covering $\tG$ is perfect.
  If there exists a non-trivial element $\tg$ of $\pi_1(G)$ satisfying $i_* (\tg) = 0 \in \pi_1(K)$, then there exists a non-trivial principal $G$-bundle $E$ over $\Sigma_h$ such that the bundle $E$ is trivial as a principal $K$-bundle.
\end{proposition}

\begin{proof}
  By the perfectness of $\tG$, we can take $\tg_j \in \tG$ ($j=1,\ldots,2h$) such that $\tg = [\tg_1, \tg_2]\cdots [\tg_{2h-1}, \tg_{2h}]$.
  Let us define a homomorphism $\rho \colon \pi_1(\Sigma_h) \to G$ by setting
  \begin{align}\label{def psi}
    \rho(a_j) = \pi(\tg_j)
  \end{align}
  for any $j$, where $\pi \colon \tG \to G$ is the universal covering.
  Then the principal $G$-bundle $G \to E_{\rho} \to \Sigma_h$ associated to the holonomy homomorphism $\rho$ is non-trivial (see \cite{Mi58}).
  We show the principal $K$-bundle $E_{i \circ \rho}$ is trivial.
  By the assumption of $\tg$, we have
  \[
    0 = i_* (\tg) = [i_*(\tg_1), i_*(\tg_2)]\cdots [i_*(\tg_{2h-1}), i_*(\tg_{2h})].
  \]
  Let us define $\widetilde{i \circ \rho} \colon \pi_1(\Sigma_h) \to \widetilde{K}$ by
  \[
    \widetilde{i \circ \rho} (a_j) = i_*(\tg_j),
  \]
  then this map $\widetilde{i \circ \rho}$ is a homomorphism satisfying $\pi \circ (\widetilde{i \circ \rho}) = i \circ \rho$, where $\pi \colon \widetilde{K} \to K$ is the universal covering.
  Thus the classifying map $B(i \circ \rho) \colon B\pi_1(\Sigma_h) \to BG \to BK$ factors into
  \[
    B(i \circ \rho) = B\pi \circ B(\widetilde{i \circ \rho}) \colon \Sigma_h \simeq B\pi_1(\Sigma_h) \to B\widetilde{K} \to BK.
  \]
  Since the fundamental group and second homotopy group of the classifying space $B\widetilde{K}$ are trivial, the map
  \[
    B(\widetilde{i \circ \rho}) \colon \Sigma_h \to B\widetilde{K}
  \]
  is null-homotopic and so is the classifying map $B(i \circ \rho)$ of the bundle $E_{i \circ \rho}$.
  Thus the bundle $E_{i \circ \rho}$ is a trivial bundle.
\end{proof}

\subsection{Contact Hamiltonian fibrations}
Let $M$ be a manifold with a contact structure $\xi$.
Let $\Cont_0(M,\xi)$ be a contact Hamiltonian diffeomorphism group, that is, the identity component of the group
\[
  \Cont(M,\xi) = \{ g \in \mathrm{Diff}(M) \mid g^*\xi = \xi \}
\]
with the $C^{\infty}$-topology.
A fiber bundle $M \to E \to B$ is called a {\it contact Hamiltonian fibration} if the structure group is reduced to the contact Hamiltonian diffeomorphism group.

The orientation preserving diffeomorphism group $\mathrm{Diff}_+(S^3)$ of the $3$-sphere is homotopy equivalent to $SO(4)$ (\cite{H}).
Hence the fundamental group $\pi_1(\mathrm{Diff}_+(S^3))$ is isomorphic to $\ZZ/2\ZZ$.
Let $\xi$ be the standard contact structure on the $3$-sphere.
The fundamental group of $\Cont_0(S^3,\xi)$ is isomorphic to $\ZZ$ (\cite{El}, \cite{CS}).
Let $i \colon \Cont_0(S^3,\xi) \hookrightarrow \mathrm{Diff}_+(S^3)$ be the inclusion, then the induced map
\[
  i_* \colon \pi_1(\Cont_0(S^3,\xi)) \cong \ZZ \to \pi_1(\mathrm{Diff}_+(S^3)) \cong \ZZ/2\ZZ
\]
is surjective (\cite{CS}).
Let $\tg \in \pi_1(\Cont_0(S^3,\xi))$ be a non-zero even number in  $\ZZ \cong \pi_1(\Cont_0(S^3,\xi))$,
then we have $i_* (\tg) = 0 \in \pi_1(\mathrm{Diff}_+(S^3))\cong \ZZ/2\ZZ$.
By the perfectness of $\widetilde{\Cont}_0(S^3,\xi)$ (\cite{Ry}) and  Proposition \ref{prop G,K bundle difference}, there is a non-trivial principal $\Cont_0(S^3,\xi)$-bundle over a closed surface that is trivial as a principal $\mathrm{Diff}_{+}(S^3)$-bundle.
In other words, there is a sphere bundle that is non-trivial as a contact Hamiltonian fibration but trivial as an oriented sphere bundle.

For a contact Hamiltonian fibration $S^3 \to E \to \Sigma_h$, let $\mathfrak{o}(E) \in H^2(\Sigma_h;\ZZ)$ be the obstruction class.
Let $\chi(E) \in \ZZ$ denote the characteristic number
\[
  \chi(E) = \langle \mathfrak{o}(E), [\Sigma_h] \rangle.
\]
\begin{proposition}\label{s3 bdl odd even}
  Let $S^3 \to E \to \Sigma_h$ be a foliated contact Hamiltonian fibration.
  If the characteristic number $\chi(E) \in \ZZ$ is even, the bundle $E$ is trivial as an oriented sphere bundle.
  If $\chi(E)$ is odd, the bundle $E$ is non-trivial as an oriented sphere bundle.
\end{proposition}
\begin{proof}

  Let $\rho \colon \pi_1(\Sigma_h) \to \Cont_0(S^3, \xi)$ be a holonomy homomorphism of $E$.
  Set $g_j = \psi(a_j)$ and take lifts $\tg_j \in \widetilde{\Cont}_0(S^3,\xi)$ of $g_j$'s, where $a_j \in \pi_1(\Sigma_h)$ are the generators.
  Set $\tg = [\tg_1, \tg_2]\cdots [\tg_{2h-1}, \tg_{2h}]$.
  Then, by the algorithm (\cite{Mi58}) for computing the characteristic number of foliated bundles, we have
  \[
    \chi(E) = [\tg_1, \tg_2]\cdots [\tg_{2g-1}, \tg_{2g}]
    = \tg \in \ZZ \cong \pi_1(\Cont_0(S^3, \xi)).
  \]
  If $\chi(E)$ is even, we have $i_*(\tg) = 0 \in \pi_1(\mathrm{Diff}_+(S^3)) \cong \ZZ/2\ZZ$.
  Thus, the bundle $E$ is trivial as an oriented sphere bundle by the same arguments in Proposition \ref{prop G,K bundle difference}.
  If $\chi(E)$ is odd, we have $i_*(\tg) = 1 \in \pi_1(\mathrm{Diff}_+(S^3)) \cong \ZZ/2\ZZ$.
  Since the characteristic number of $E$ is non-zero, the bundle $E$ is non-trivial as an oriented sphere bundle.
\end{proof}

\subsection{Hamiltonian fibrations}
Let $M$ be a manifold with a symplectic form $\omega$.
A fiber bundle $M \to E \to B$ is called a {\it Hamiltonian fibration} if the structure group is reduced to the Hamiltonian diffeomorphism group $\Ham(M,\omega)$.

Let us consider the $4$-manifold $S^2 \times S^2$.

By Propositions \ref{prop G,K bundle difference} and \ref{ostrover loop}, we obtain the following:

\begin{proposition}\label{ham s2s2}
  There exists a positive integer $h_0$ and a non-trivial Hamiltonian fibration $p_0\colon E_0 \to \Sigma_{h_0}$ over a closed surface.
\end{proposition}

We can also prove that the Hamiltonian fibration $p_0\colon E_0 \to \Sigma_{h_0}$ in Proposition \ref{ham s2s2} is stably non-trivial in the following sense.

\begin{proposition}\label{stably nt}
Let $(N,\omega_N)$ be a closed symplectic manifold and $p \colon \epsilon_N = \Sigma_{h_0}\times N \to \Sigma_{h_0}$ the trivial $N$-bundle.
Then, the Whitney sum
\[
  E_0 \oplus \epsilon_N \to \Sigma_{h_0}
\]
is non-trivial as a Hamiltonian fibration.
\end{proposition}

To prove Proposition \ref{stably nt}, we 
use the following theorem essentially proved by Entov and Polterovich.

\begin{theorem}[Theorem 5.1 of \cite{EP09}]\label{spectral invariant of product}
Let $(N,\omega_N)$ be a closed symplectic manifold.
For $\lambda\geq1$, let $\omega_{\lambda,N}$ denote the symplectic form $\mathrm{pr}_1^\ast\omega_\lambda+\mathrm{pr}_2^\ast\omega_N$ where $\mathrm{pr}_1\colon S^2\times S^2\times N\to S^2\times S^2$, $\mathrm{pr}_2\colon S^2\times S^2\times N\to N$ are the first, second projection, respectively.
Then, there exists a function $\mu^{\lambda,N}\colon\tHam(S^2\times S^2\times N,\omega_{\lambda,N})\to\RR$ such that
\[\mu^{\lambda,N}(\tilde\phi_N)=\mu^\lambda(\tilde\phi)\]
for every $\tilde\phi\in\tHam(S^2\times S^2,\omega_{\lambda,N})$.

Here, $\tilde\phi_N$ is the element of $\tHam(S^2\times S^2\times N,\omega_{\lambda,N})$ represented by the path $\{\phi_N^t\}_{t\in[0,1]}$ defined by $\phi_N^t(x,y)=\left(\phi^t(x),y\right)$ where $\{\phi^t\}_{t\in[0,1]}$ is a path in $\Ham(S^2 \times S^2,\omega_\lambda)$ representing $\tilde{\phi}$.
\end{theorem}

\begin{remark}
The function $\mu^{\lambda,N}$ satisfy the conditions of  ``partial Calabi quasi-morphism'' (\cite[Theorem 3.2]{E}).
However, the authors do not know whether the restriction of $\mu^{\lambda,N}$ to the fundamental group is homomorphism or not.
\end{remark}

\begin{proof}[Proof of Proposition $\ref{stably nt}$]
Let $\tg = \{ \tg^t \}_{t \in [0,1]}$ be a path in $\Ham(S^2\times S^2,\omega_\lambda)$ corresponding to the bundle $E_0$.
Define a loop $\tg_N = \{ \tg_N^t \}_{t \in [0,1]}$ in $\Ham(S^2\times S^2\times N,\omega_{\lambda,N})$ by $\tg_N^t(x,y)=\left(g^t(x),y\right)$.
Then, by Theorem \ref{spectral invariant of product} and Proposition \ref{ostrover loop}, we have that $\mu^{\lambda,N}(\tg_N) = \mu^{\lambda}(\tg) \neq 0$, in particular, $\tg_N$ is a non-trivial element of $\pi_1\left(\Ham(S^2\times S^2\times N,\omega_{\lambda,N})\right)$.
By Proposition \ref{prop G,K bundle difference}, the proposition follows.
\end{proof}



\bibliographystyle{amsalpha}
\bibliography{second_cohom_of_ham_kyocho.bib}

\providecommand{\bysame}{\leavevmode\hbox to3em{\hrulefill}\thinspace}
\providecommand{\MR}{\relax\ifhmode\unskip\space\fi MR }
\providecommand{\MRhref}[2]{%
  \href{http://www.ams.org/mathscinet-getitem?mr=#1}{#2}
}
\providecommand{\href}[2]{#2}
\begin{thebibliography}{CMPSC11}

\bibitem[Anj02]{A}
S\'{\i}lvia Anjos, \emph{Homotopy type of symplectomorphism groups of
  {$S^2\times S^2$}}, Geom. Topol. \textbf{6} (2002), 195--218.

\bibitem[Ban78]{Ba78}
Augustin Banyaga, \emph{Sur la structure du groupe des diff\'{e}omorphismes qui
  pr\'{e}servent une forme symplectique}, Comment. Math. Helv. \textbf{53}
  (1978), no.~2, 174--227.

\bibitem[Ban97]{Ba97}
\bysame, \emph{The structure of classical diffeomorphism groups}, Mathematics
  and its Applications, vol. 400, Kluwer Academic Publishers Group, Dordrecht,
  1997.

\bibitem[Ben91]{benson_1991}
D.~J. Benson, \emph{Representations and cohomology}, Cambridge Studies in
  Advanced Mathematics, vol.~2, Cambridge University Press, 1991.

\bibitem[BG92]{BG}
J.~Barge and \'{E}. Ghys, \emph{Cocycles d'{E}uler et de {M}aslov}, Math. Ann.
  \textbf{294} (1992), no.~2, 235--265.

\bibitem[Bor13]{B}
Matthew~Strom Borman, \emph{Quasi-states, quasi-morphisms, and the moment map},
  Int. Math. Res. Not. IMRN (2013), no.~11, 2497--2533.

\bibitem[Bou95]{MR1338286}
Abdessalam Bouarich, \emph{Suites exactes en cohomologie born\'{e}e r\'{e}elle
  des groupes discrets}, C. R. Acad. Sci. Paris S\'{e}r. I Math. \textbf{320}
  (1995), no.~11, 1355--1359.

\bibitem[Bra15]{Br}
Michael Brandenbursky, \emph{Bi-invariant metrics and quasi-morphisms on groups
  of {H}amiltonian diffeomorphisms of surfaces}, Internat. J. Math. \textbf{26}
  (2015), no.~9, 1550066, 29.

\bibitem[Bro82]{Bro}
Kenneth~S. Brown, \emph{Cohomology of groups}, Graduate Texts in Mathematics,
  vol.~87, Springer-Verlag, New York-Berlin, 1982.

\bibitem[BZ15]{BZ}
Matthew~Strom Borman and Frol Zapolsky, \emph{Quasimorphisms on
  contactomorphism groups and contact rigidity}, Geom. Topol. \textbf{19}
  (2015), no.~1, 365--411.

\bibitem[Cal04]{Cal04}
Danny Calegari, \emph{Circular groups, planar groups, and the {E}uler class},
  Proceedings of the {C}asson {F}est, Geom. Topol. Monogr., vol.~7, Geom.
  Topol. Publ., Coventry, 2004, pp.~431--491.

\bibitem[Cal09]{Cal}
\bysame, \emph{scl}, MSJ Memoirs, vol.~20, Mathematical Society of Japan,
  Tokyo, 2009.

\bibitem[Cas17]{C}
Alexander~Caviedes Castro, \emph{Calabi quasimorphisms for monotone coadjoint
  orbits}, J. Topol. Anal. \textbf{9} (2017), no.~4, 689--706.

\bibitem[CMPSC11]{Ch}
Indira Chatterji, Guido Mislin, Christophe Pittet, and Laurent Saloff-Coste,
  \emph{A geometric criterion for the boundedness of characteristic classes},
  Math. Ann. \textbf{351} (2011), no.~3, 541--569.

\bibitem[CS16]{CS}
Roger Casals and Old\v{r}ich Sp\'{a}\v{c}il, \emph{Chern-{W}eil theory and the
  group of strict contactomorphisms}, J. Topol. Anal. \textbf{8} (2016), no.~1,
  59--87.

\bibitem[Dup78]{D}
Johan~L. Dupont, \emph{Curvature and characteristic classes}, Lecture Notes in
  Mathematics, Vol. 640, Springer-Verlag, Berlin-New York, 1978.

\bibitem[Eli92]{El}
Yakov Eliashberg, \emph{Contact {$3$}-manifolds twenty years since {J}.
  {M}artinet's work}, Ann. Inst. Fourier (Grenoble) \textbf{42} (1992),
  no.~1-2, 165--192.

\bibitem[Ent14]{E}
Michael Entov, \emph{Quasi-morphisms and quasi-states in symplectic topology},
  Proceedings of the {I}nternational {C}ongress of {M}athematicians---{S}eoul
  2014. {V}ol. {II}, Kyung Moon Sa, Seoul, 2014, pp.~1147--1171.

\bibitem[EP03]{EP03}
Michael Entov and Leonid Polterovich, \emph{Calabi quasimorphism and quantum
  homology}, Int. Math. Res. Not. (2003), no.~30, 1635--1676.

\bibitem[EP09]{EP09}
\bysame, \emph{Rigid subsets of symplectic manifolds}, Compos. Math.
  \textbf{145} (2009), no.~3, 773--826.

\bibitem[FOOO19]{FOOO}
Kenji Fukaya, Yong-Geun Oh, Hiroshi Ohta, and Kaoru Ono, \emph{Spectral
  invariants with bulk, quasi-morphisms and {L}agrangian {F}loer theory}, Mem.
  Amer. Math. Soc. \textbf{260} (2019), no.~1254, x+266.

\bibitem[FPR18]{FPR18}
Maia Fraser, Leonid Polterovich, and Daniel Rosen, \emph{On {S}andon-type
  metrics for contactomorphism groups}, Ann. Math. Qu\'{e}. \textbf{42} (2018),
  no.~2, 191--214.

\bibitem[Fri17]{Fr}
Roberto Frigerio, \emph{Bounded cohomology of discrete groups}, Mathematical
  Surveys and Monographs, vol. 227, American Mathematical Society, Providence,
  RI, 2017.

\bibitem[Gei08]{G}
Hansj\"{o}rg Geiges, \emph{An introduction to contact topology}, Cambridge
  Studies in Advanced Mathematics, vol. 109, Cambridge University Press,
  Cambridge, 2008.

\bibitem[GG04]{GG}
Jean-Marc Gambaudo and \'{E}tienne Ghys, \emph{Commutators and diffeomorphisms
  of surfaces}, Ergodic Theory Dynam. Systems \textbf{24} (2004), no.~5,
  1591--1617.

\bibitem[Ghy01]{MR1876932}
\'{E}tienne Ghys, \emph{Groups acting on the circle}, Enseign. Math. (2)
  \textbf{47} (2001), no.~3-4, 329--407.

\bibitem[Giv90]{Gi}
A.~B. Givental, \emph{Nonlinear generalization of the {M}aslov index}, Theory
  of singularities and its applications, Adv. Soviet Math., vol.~1, Amer. Math.
  Soc., Providence, RI, 1990, pp.~71--103.

\bibitem[GKT11]{GKT}
\'{S}wiatos\l~aw Gal, Jarek K\k{e}dra, and Aleksy Tralle, \emph{On the
  algebraic independence of {H}amiltonian characteristic classes}, J.
  Symplectic Geom. \textbf{9} (2011), no.~1, 1--9.

\bibitem[Gro82]{Gr82}
Michael Gromov, \emph{Volume and bounded cohomology}, Inst. Hautes \'{E}tudes
  Sci. Publ. Math. (1982), no.~56, 5--99 (1983).

\bibitem[Ham65]{Hams}
Mary-Elizabeth Hamstrom, \emph{The space of homeomorphisms on a torus},
  Illinois J. Math. \textbf{9} (1965), 59--65.

\bibitem[Hat83]{H}
Allen~E. Hatcher, \emph{A proof of the {S}male conjecture, {${\rm
  Diff}(S^{3})\simeq {\rm O}(4)$}}, Ann. of Math. (2) \textbf{117} (1983),
  no.~3, 553--607.

\bibitem[Ish14]{I}
Tomohiko Ishida, \emph{Quasi-morphisms on the group of area-preserving
  diffeomorphisms of the 2-disk via braid groups}, Proc. Amer. Math. Soc. Ser.
  B \textbf{1} (2014), 43--51.

\bibitem[JK02]{JK}
Tadeusz Januszkiewicz and Jarek K\k{e}dra, \emph{Characteristic classes of
  smooth fibrations}, 2002.

\bibitem[KK19]{KK}
Morimichi Kawasaki and Mitsuaki Kimura, \emph{$\hat{G}$-invariant
  quasimorphisms and symplectic geometry of surfaces}, to appear in
  \emph{Israel J.\ Math}, arXiv:1911.10855v2 (2019).

\bibitem[KKMM20]{KKMM20}
Morimichi Kawasaki, Mitsuaki Kimura, Takahiro Matsushita, and Masato Mimura,
  \emph{Bavard's duality theorem for mixed commutator length},
  arXiv:2007.02257v3, to appear in Enseign. Math. (2020).

\bibitem[KKMM21]{KKMM21}
Morimichi Kawasaki, Mitsuaki Kimura, Takahiro Matsushita, and Masato Mimura,
  \emph{Commuting symplectomorphisms on a surface and the flux homomorphism},
  preprint, arXiv:2102.12161v1 (2021).

\bibitem[KR11]{KR}
Agnieszka Kowalik and Tomasz Rybicki, \emph{On the homeomorphism groups of
  manifolds and their universal coverings}, Cent. Eur. J. Math. \textbf{9}
  (2011), no.~6, 1217--1231.

\bibitem[Man20]{Mann20}
Kathryn Mann, \emph{Unboundedness of some higher {E}uler classes}, Algebr.
  Geom. Topol. \textbf{20} (2020), no.~3, 1221--1234.

\bibitem[Mar20]{M}
Shuhei Maruyama, \emph{The dixmier-douady class, the action homomorphism, and
  group cocycles on the symplectomorphism group}, 2020.

\bibitem[Mar22]{MR4330215}
\bysame, \emph{Extensions of quasi-morphisms to the symplectomorphism group of
  the disk}, Topology Appl. \textbf{305} (2022), Paper No. 107880.

\bibitem[McD04]{Mc04}
Dusa McDuff, \emph{Lectures on groups of symplectomorphisms}, Rend. Circ. Mat.
  Palermo (2) Suppl. (2004), no.~72, 43--78.

\bibitem[McD10]{Mc10}
\bysame, \emph{Monodromy in {H}amiltonian {F}loer theory}, Comment. Math. Helv.
  \textbf{85} (2010), no.~1, 95--133.

\bibitem[Mil58]{Mi58}
John Milnor, \emph{On the existence of a connection with curvature zero},
  Comment. Math. Helv. \textbf{32} (1958), 215--223.

\bibitem[MN21]{MN21}
Nicolas Monod and Sam Nariman, \emph{On the bounded cohomology of certain
  homeomorphism groups}, preprint, arXiv:2111.04365 (2021).

\bibitem[MR18]{MR}
Kathryn Mann and Christian Rosendal, \emph{Large-scale geometry of
  homeomorphism groups}, Ergodic Theory Dynam. Systems \textbf{38} (2018),
  no.~7, 2748--2779.

\bibitem[NSW08]{MR2392026}
J\"{u}rgen Neukirch, Alexander Schmidt, and Kay Wingberg, \emph{Cohomology of
  number fields}, second ed., Grundlehren der Mathematischen Wissenschaften
  [Fundamental Principles of Mathematical Sciences], vol. 323, Springer-Verlag,
  Berlin, 2008.

\bibitem[Oh05]{Oh05}
Yong-Geun Oh, \emph{Construction of spectral invariants of {H}amiltonian paths
  on closed symplectic manifolds}, The breadth of symplectic and {P}oisson
  geometry, Progr. Math., vol. 232, Birkh\"{a}user Boston, Boston, MA, 2005,
  pp.~525--570.

\bibitem[Ost06]{Os}
Yaron Ostrover, \emph{Calabi quasi-morphisms for some non-monotone symplectic
  manifolds}, Algebr. Geom. Topol. \textbf{6} (2006), 405--434.

\bibitem[OT09]{OT}
Yaron Ostrover and Ilya Tyomkin, \emph{On the quantum homology algebra of toric
  {F}ano manifolds}, Selecta Math. (N.S.) \textbf{15} (2009), no.~1, 121--149.

\bibitem[Pon86]{P}
L.~S. Pontryagin, \emph{Selected works vol.~2, topological groups}, Translated
  from the Russian and with a preface by A. Brown. With additional material
  translated by P. S. V. Naidu., Gordon and Breach Science Publishers, Inc.,
  New York-London-Paris, 1986.

\bibitem[PR14]{PR}
Leonid Polterovich and Daniel Rosen, \emph{Function theory on symplectic
  manifolds}, CRM Monograph Series, vol.~34, American Mathematical Society,
  Providence, RI, 2014.

\bibitem[Py06a]{Py6-2}
Pierre Py, \emph{Quasi-morphismes de {C}alabi et graphe de {R}eeb sur le tore},
  C. R. Math. Acad. Sci. Paris \textbf{343} (2006), no.~5, 323--328.

\bibitem[Py06b]{Py6-1}
\bysame, \emph{Quasi-morphismes et invariant de {C}alabi}, Ann. Sci. \'{E}cole
  Norm. Sup. (4) \textbf{39} (2006), no.~1, 177--195.

\bibitem[Rez97]{R}
Alexander~G. Reznikov, \emph{Characteristic classes in symplectic topology},
  Selecta Math. (N.S.) \textbf{3} (1997), no.~4, 601--642, Appendix D by Ludmil
  Katzarkov.

\bibitem[Ryb10]{Ry}
Tomasz Rybicki, \emph{Commutators of contactomorphisms}, Adv. Math.
  \textbf{225} (2010), no.~6, 3291--3326.

\bibitem[Sch00]{Sch}
Matthias Schwarz, \emph{On the action spectrum for closed symplectically
  aspherical manifolds}, Pacific J. Math. \textbf{193} (2000), no.~2, 419--461.

\bibitem[She14]{Sh}
Egor Shelukhin, \emph{The action homomorphism, quasimorphisms and moment maps
  on the space of compatible almost complex structures}, Comment. Math. Helv.
  \textbf{89} (2014), no.~1, 69--123.

\bibitem[Sht16]{Sht}
A.~I. Shtern, \emph{Extension of pseudocharacters from normal subgroups,
  {III}}, Proc. Jangjeon Math. Soc. \textbf{19} (2016), no.~4, 609--614.

\bibitem[Sim07]{GBS}
Gabi~Ben Simon, \emph{The nonlinear {M}aslov index and the {C}alabi
  homomorphism}, Commun. Contemp. Math. \textbf{9} (2007), no.~6, 769--780.
  \MR{2372458}

\bibitem[SS20]{SS}
Yasha Savelyev and Egor Shelukhin, \emph{K-theoretic invariants of
  {H}amiltonian fibrations}, J. Symplectic Geom. \textbf{18} (2020), no.~1,
  251--289.

\bibitem[Thu72]{MR298692}
William Thurston, \emph{Noncobordant foliations of {$S^{3}$}}, Bull. Amer.
  Math. Soc. \textbf{78} (1972), 511--514.

\bibitem[Thu74]{MR339267}
\bysame, \emph{Foliations and groups of diffeomorphisms}, Bull. Amer. Math.
  Soc. \textbf{80} (1974), 304--307.

\bibitem[Ush11]{U}
Michael Usher, \emph{Deformed {H}amiltonian {F}loer theory, capacity estimates
  and {C}alabi quasimorphisms}, Geom. Topol. \textbf{15} (2011), no.~3,
  1313--1417.

\bibitem[Via18]{V18}
Renato Vianna, \emph{Continuum families of non-displaceable {L}agrangian tori
  in {$(\Bbb CP^1)^{2m}$}}, J. Symplectic Geom. \textbf{16} (2018), no.~3,
  857--883.

\bibitem[Whi78]{W}
George~W. Whitehead, \emph{Elements of homotopy theory}, Graduate Texts in
  Mathematics, vol.~61, Springer-Verlag, New York-Berlin, 1978.

\bibitem[Woo71]{Wo71}
John~W. Wood, \emph{Bundles with totally disconnected structure group},
  Comment. Math. Helv. \textbf{46} (1971), 257--273.

\end{thebibliography}

\end{document}